\documentclass[10pt]{article}
\usepackage{lmodern}
\usepackage{amsmath}
\usepackage[T1]{fontenc}
\usepackage[utf8]{inputenc}
\usepackage{authblk}
\usepackage{amsfonts}
\usepackage{graphicx}
\usepackage{rotating}
\usepackage{amssymb}
\usepackage[english]{babel}
\usepackage{color}
\usepackage{amsthm}
\usepackage{graphicx}
\usepackage{mathrsfs}
\usepackage{makecell}
\usepackage{microtype}
\usepackage{mathscinet}
\usepackage{array}
\usepackage{multirow}
\usepackage{enumerate}
\usepackage[cal=boondoxo,bb=ams]{mathalfa}
\usepackage{hyperref}
\hypersetup{hidelinks}
\usepackage{booktabs}

\newtheorem{theorem}{Theorem}[section]
\newtheorem{prop}{Proposition}[section]
\newtheorem{lemma}{Lemma}[section]
\newtheorem{coro}{Corollary}[section]
\newtheorem{remark}{Remark}[section]
\newtheorem{exam}{Example}[section]

\newcommand{\ml}{\mathcal}
\newcommand{\mb}{\mathbb}
\DeclareMathOperator{\non}{non}
\DeclareMathOperator{\lin}{lin}
\DeclareMathOperator{\intt}{int}
\DeclareMathOperator{\extt}{ext}
\DeclareMathOperator{\bdd}{bdd}

\title{Large time behavior for the nonlinear dissipative Boussinesq equation}
\author[1]{Wenhui Chen\thanks{Wenhui Chen (wenhui.chen.math@gmail.com)}}
\affil[1]{School of Mathematics and Information Science, Guangzhou University, 510006 Guangzhou, China}
\author[2]{Hiroshi Takeda\thanks{Hiroshi Takeda (h-takeda@fit.ac.jp)}}
\affil[2]{Department of Intelligent Mechanical Engineering, Faculty of Engineering, Fukuoka Institute of Technology,  811-0295 Fukuoka, Japan}
\date{}

\begin{document}

		\maketitle
		\begin{abstract}
			\medskip
	In this paper, we study the nonlinear dissipative Boussinesq equation in the whole space $\mathbb{R}^n$ with $L^1$ integrable data. As our preparations, the optimal estimates as well as the optimal leading terms for the linearized model are derived by performing the WKB analysis and the Fourier analysis. Then, under some conditions on the power $p$ of nonlinearity, we demonstrate global (in time) existence of small data Sobolev solutions with different regularities to the nonlinear model by applying some fractional order interpolations, where the optimal growth ($n=2$) and decay ($n\geqslant 3$) estimates of solutions for large time are given. Simultaneously, we get a new large time asymptotic profile of global (in time) solutions. These results imply some influence of dispersion and dissipation on qualitative properties of solution.
			\\
			
			
			\noindent\textbf{AMS Classification (2020)} 35G25, 35A01, 35B40 
		\end{abstract}
\fontsize{12}{15}
\selectfont

\section{Introduction}\label{Section-Introduction}\setcounter{equation}{0}
\subsection{Background of the nonlinear (dissipative) Boussinesq equations}
$\ \ \ \ $Since the last two centuries, the propagation of water waves over fluids has been widely considered in fluid mechanics, hydrogeology and coastal engineering. Some mathematical models are established to simulate it in different circumstances. Among them,   the 
innovative work \cite{Boussinesq-03} from J.V. Boussinesq  built the fundamental model for the propagation of long waves with small amplitude on the surface of shallow water, that is the so-called Boussinesq equation as follows:
\begin{align*}
	u_{tt}-u_{xx}+u_{xxxx}=(u^2)_{xx},
\end{align*}
in which the unknown function $u=u(t,x)\in\mb{R}$ denotes an elevation of the free surface of a fluid. This work \cite{Boussinesq-03} is innovative to give some scientific explanations to the phenomenon of permanent waves found by \cite{Russell}. Soon afterwards, a general Boussinesq equation in the sense of equipping the general nonlinearity  $[f(u;p)]_{xx}$ with a smooth power type function $f(u;p)$ was considered in the modeling of nonlinear strings \cite{Bona-Sachs=1988}, whose multi-dimensional model is written by
\begin{align}\label{Generalized_Boussinesq_Eq}
	u_{tt}-\Delta u+\Delta^2u=\Delta f(u;p).
\end{align}
The power type nonlinearity behaves as $f(u;p)=O(|u|^p)$ when $u\to 0$ with $p>1$. The Boussinesq equation \eqref{Generalized_Boussinesq_Eq} also can model nonlinear waves phenomena in the media with dispersion. Nowadays, qualitative properties of solutions to the Boussinesq equation \eqref{Generalized_Boussinesq_Eq} with various power type nonlinearities have been deeply studied, including well-posedness, blow-up and asymptotic behavior. We refer the interested reader to \cite{Levine=1974,Deift-Tomei-Trub=1982,Bona-Sachs=1988,Liu=1997,Cho-Ozawa=2007} and references given therein.

Because some resistances can not be neglected in real world problems, we naturally consider the PDEs in physic with dissipation, such as Navier-Stokes equations (viscous dissipation), thermoelastic equations (thermal dissipation), magnetohydrodynamics equations (magnetic dissipation). Let us turn to the Boussinesq equation \eqref{Generalized_Boussinesq_Eq} with dissipation. Actually, the presence of dissipation generally destroys the balance between the dispersion from the linear semigroup $\exp(\pm i|D|\langle D\rangle t)$ and the nonlinearity $\Delta f(u;p)$. Here, the symbol of pseudo-differential operator $|D|$ is denoted by $|\xi|$, and the Japanese bracket is marked by $\langle\xi\rangle:=\sqrt{1+|\xi|^2}$. In our consideration, provided that an energy production is allowed in the model, another form of the Boussinesq paradigm can be encountered. To be specific, the nonlinear dissipative Boussinesq equation arises in the literature (cf. \cite{Christov-Velarde=1994,Christov=2002,Christov-Maugin-Porubov=2007})  to describe the flow in thin viscous liquid layers or  the oscillation of nonlinear elastic beams, whose multi-dimensional version is addressed by
\begin{align}\label{Eq-Intro-Dissipation-Boussinesq}
	u_{tt}-\Delta u+\Delta^2u+2\mu\Delta^2u_t=\Delta f(u;p)
\end{align}
with the dissipation coefficient $\mu\in(0,1)$. This model may show the balance between the energy input and the dissipation modulated by the presence of nonlinearity and dispersion. To understand this balance and underlying physical phenomena, it is interesting to characterize some influence of dispersion and dissipation on qualitative properties of solutions.

Let us recall some recent researches for the dissipative Boussinesq equation \eqref{Eq-Intro-Dissipation-Boussinesq}. Firstly, the initial boundary value problem of \eqref{Eq-Intro-Dissipation-Boussinesq} has been studied mathematically by \cite{Wang-Su=2016,Wang-Su=2016-02}, in which global (in time) existence associated with large time behavior of solutions and blow-up of solutions were found. Later, the authors of \cite{Wang-Su=2019} investigated solvabilities and blow-up of solutions for the Cauchy problem of \eqref{Eq-Intro-Dissipation-Boussinesq} with the nonlinearity $f(u;p)=\pm |u|^{p-1}u$. The recent paper \cite{Su-Wang=2020} demonstrated global (in time) existence of Sobolev solutions to the Cauchy problem of \eqref{Eq-Intro-Dissipation-Boussinesq} with the nonlinearity satisfying $f(u;p)=O(|u|^p)$ when $u\to 0$, where its initial data belonged to suitable Riesz potential spaces of negative order, particularly, $u_1\in \dot{H}^{s-2}\cap \dot{H}^{-2}_q$ with  $q\in[1,2]$. As their preparations, \cite{Su-Wang=2020} also derived some decay estimates and asymptotic profiles of solutions to the corresponding linearized Cauchy problem. Concerning other studies on the Boussinesq equation with different dissipations, we refer to \cite{Varlamov-199601,Varlamov-199602,Ikehata-Soga=2015,Liu-Wang=2019,Liu-Wang=2020,Chen-Dao=2022} and references therein.

To the best of authors' knowledge, large time behavior (optimal growth/decay estimates and asymptotic profiles) for the nonlinear dissipative Boussinesq equation \eqref{Eq-Intro-Dissipation-Boussinesq} in $\mb{R}^n$ has never been studied in the literature, even for its corresponding linearized Cauchy problem with $L^1$ integrable data. We will answer these questions in the present work. 

\subsection{Main contributions of this paper}
$\ \ \ \ $In the present paper, we mainly consider the Cauchy problem for the nonlinear Boussinesq equation with strong dissipation as follows:
\begin{align}\label{Eq-NonLinear-Dissipative-Boussinesq}
	\begin{cases}
		u_{tt}-\Delta u+\Delta^2u+2\mu\Delta^2u_t=\Delta f(u;p),&x\in\mb{R}^n,\ t>0,\\
		u(0,x)=u_0(x),\ u_t(0,x)=u_1(x),&x\in\mb{R}^n,
	\end{cases}
\end{align}
with the dissipation coefficient $\mu\in(0,1)$, where the smooth function $f(u;p)$ behaves as a power type nonlinearity such that $f(u;p)=O(|u|^p)$ when $u\to0$ with $p>1$. We are interested in exploring some large time behavior of global (in time) solutions to the nonlinear Cauchy problem \eqref{Eq-NonLinear-Dissipative-Boussinesq} with $L^1$ initial integrable data.

As our preparations, by performing the WKB analysis and the Fourier analysis, we firstly study the linearized dissipative Boussinesq equation \eqref{Eq-Linear-Dissipative-Boussinesq} in Section \ref{Section-Linearized-Model}, whose solution $v=v(t,x)$ fulfills the following optimal estimates:
\begin{align}\label{Dnt}
	\|v(t,\cdot)\|_{L^2}\simeq \ml{D}_n(t):=\begin{cases}
		\sqrt{t}&\mbox{when}\ \ n=1\\
		\sqrt{\ln t}&\mbox{when}\ \ n=2\\
		t^{-\frac{n-2}{8}}&\mbox{when}\ \ n\geqslant3
	\end{cases}\ \ \mbox{and}\ \ \|v(t,\cdot)\|_{\dot{H}^{s}}\simeq t^{-\frac{2s+n-2}{8}}
\end{align}
for large time $t\gg1$, provided that $|\int_{\mb{R}^n}v_1(x)\mathrm{d}x|\neq 0$, with $n\geqslant 1$ and $s\geqslant 1$. Hereafter,  $h\simeq  g$ holds when $h\lesssim g\lesssim h$. Particularly, we discover the critical dimension $n=2$ for distinguishing the decisive role between dispersion (wave effect) and dissipation (damping effect) on large time behavior of solution in the $L^2$ norm. Furthermore,  by taking 
 \begin{align}\label{G0-G1-fun}
 	\ml{G}_0(t,x)&:=\ml{F}^{-1}_{\xi\to x}\left[\cos\left(|\xi|t+\frac{1}{2}|\xi|^3t\right)\mathrm{e}^{-\mu|\xi|^4t}\right],\notag\\  \ml{G}_1(t,x)&:=\ml{F}^{-1}_{\xi\to x}\left[\frac{1}{|\xi|}\sin\left(|\xi|t+\frac{1}{2}|\xi|^3t\right)\mathrm{e}^{-\mu|\xi|^4t}\right],
 \end{align}
we introduce  the functions listed below to be the large time profiles:
\begin{align}
	v^{(1,p)}(t,x)&:=\ml{G}_1(t,x)P_{v_1},\label{Profile-v1p}\\
	v^{(2,p)}(t,x)&:=\ml{G}_0(t,x)P_{v_0}-\nabla\ml{G}_1(t,x)\circ M_{v_1}-\frac{t}{8}\Delta^2\ml{G}_0(t,x)P_{v_1},\label{Profile-v2p}
\end{align}
with the symbol $\circ$ to be the inner product, where $P_{g}:=\int_{\mb{R}^n}g(x)\mathrm{d}x$ and $M_g:=\int_{\mb{R}^n}xg(x)\mathrm{d}x$. Our results show that $v^{(1,p)}=v^{(1,p)}(t,x)$ is the optimal leading term of solution in the sense of
\begin{align*}
\|v(t,\cdot)-v^{(1,p)}(t,\cdot)\|_{\dot{H}^s}\simeq t^{-\frac{2s+n}{8}}
\end{align*}
for large time $t\gg1$, provided that $|P_{v_1}|+|P_{v_0-\frac{2s+n}{64\mu}v_1}|+|M_{v_1}|\neq0$, with $n\geqslant 1$ and $s\geqslant0$. Moreover, $v^{(2,p)}=v^{(2,p)}(t,x)$ is a second profile of solution in the sense of
\begin{align*}
\|v(t,\cdot)-v^{(1,p)}(t,\cdot)-v^{(2,p)}(t,\cdot)\|_{\dot{H}^s}=o(t^{-\frac{2s+n}{8}})
\end{align*}
for large time $t\gg1$. These results imply some new sharp large time behavior of solution to the linearized dissipative Boussinesq equation \eqref{Eq-Linear-Dissipative-Boussinesq}.

Thanks to the derived optimal estimates for the linearized Cauchy problem, we construct suitable time-weighted Sobolev spaces and demonstrate a contraction mapping uniformly in time. Under some conditions on the exponent $p$ in the nonlinearity, by using some fractional order interpolations from the harmonic analysis, the global (in time) existence of small data Sobolev solutions, to the nonlinear dissipative Boussinesq equation \eqref{Eq-NonLinear-Dissipative-Boussinesq} is proved. Particularly, in the limit case $n=2$ with $p=4$, we cannot get uniformly (in time) bounded estimates in the usual  $H^s$ framework due to the logarithmic growth estimate for the linearized model, i.e. $\|v(t,\cdot)\|_{L^2}\simeq\sqrt{\ln t}$ as $t\gg1$ for $n=2$. To overcome this difficulty and retain this optimal growth rate, we introduce suitable solution spaces with additional $L^6$ and $L^{\infty}$ regularities for $n=2$ with $p=4$.  Next, we derive a new asymptotic profile of global (in time) solution to the nonlinear Cauchy problem \eqref{Eq-NonLinear-Dissipative-Boussinesq}, which is precisely given by
\begin{align}\label{Nonlinear-Profile}
u^{(1,p)}(t,x):=\ml{G}_1(t,x)P_{u_1}.
\end{align}
Additionally, the optimal growth ($n=2$) and decay ($n\geqslant 3$) estimates of solutions are derived
\begin{align}\label{Optimal-Est-Nonlinear}
	\|u(t,\cdot)\|_{L^2}\simeq \ml{D}_n(t)\ \ \mbox{and}\ \ \|u(t,\cdot)\|_{\dot{H}^{s}}\simeq t^{-\frac{2s+n-2}{8}}
\end{align}
for large time $t\gg1$, provided that $|\int_{\mb{R}^n}u_1(x)\mathrm{d}x|\neq 0$, with $n\geqslant 2$ and $s\geqslant 1$. The large time profile $u^{(1,p)}=u^{(1,p)}(t,x)$ explains some influence from dispersive and dissipation on asymptotic behavior of global (in time) solutions.


\section{Main results}\label{Section-Main-Results}\setcounter{equation}{0}
$\ \ \ \ $We state our first result on the global (in time) existence of small data Sobolev solution to the nonlinear dissipative Boussinesq equation \eqref{Eq-NonLinear-Dissipative-Boussinesq} with different regularities (depending on the parameter $s$) on initial data. Hereafter, we denote by $\lceil a\rceil:=\min\{A\in\mb{Z}:a\leqslant A\}$ the ceiling function, and we take $(a)_+:=\max\{a,0\}$ with $\frac{1}{(a)_+}:=\infty$ when $a\leqslant 0$.
\begin{theorem}\label{Thm-GESDS}
Let $p>1$ fulfilling one of the following conditions.
\begin{itemize}
	\item If $s\in[1,2]$, we assume  $\max\{2,\frac{n+2}{n-1}\}\leqslant p\leqslant\frac{n}{(n-2s)_+}$ when $n\geqslant 2$.
	\item If $s\in(2,\frac{n}{2}]$, we assume  $\lceil s-2\rceil+1< p\leqslant \frac{n-2s+4}{(n-2s)_+}$ when $n\geqslant 5$.
	\item If $s\in(\frac{n}{2},\infty)$, we assume $p\geqslant \max\{2,\frac{n+2}{n-1}\}$ and $p>s-1$ when $n\geqslant2$.
\end{itemize}
Then, there exists a sufficiently small constant $\epsilon_0>0$ such that for any initial data
\begin{align}\label{Data-Space}
	(u_0,u_1)\in\ml{A}_s:=(H^s\cap L^1)\times (H^{(s-4)_+}\cap L^1)
\end{align}
satisfying $\|(u_0,u_1)\|_{\ml{A}_s}\leqslant\epsilon_0$, the nonlinear dissipative Boussinesq equation \eqref{Eq-NonLinear-Dissipative-Boussinesq} admits a unique global (in time) Sobolev solution
\begin{align*}
	u\in\begin{cases}
		\ml{C}([0,T],H^s\cap L^6)&\mbox{for}\ \ n=2\ \ \mbox{with}\ \ p=4 \ \ \mbox{when}\ \ s\in[1,3),\\
		\ml{C}([0,T],H^s\cap L^6\cap L^{\infty})&\mbox{for}\ \ n=2\ \ \mbox{with}\ \ p=4 \ \ \mbox{when}\ \ s\in[3,\infty),\\
		\ml{C}([0,T],H^s)&\mbox{otherwise}.
	\end{cases}
\end{align*} 
Furthermore, the next estimates hold:
\begin{align*}
\|u(t,\cdot)\|_{L^2}&\lesssim \ml{D}_n(1+t)\|(u_0,u_1)\|_{\ml{A}_s},\\
\|u(t,\cdot)\|_{\dot{H}^s}&\lesssim (1+t)^{-\frac{2s+n-2}{8}}\|(u_0,u_1)\|_{\ml{A}_s}.
\end{align*}
\end{theorem}

\begin{exam}We state some examples for the suitable range of $p$ to guarantee global (in time) existence of solution with small initial data.
	\begin{itemize}
		\item When $n=2$, we take $p\geqslant 4$ and $p>s-1$ for $s\in[1,\infty)$.
		\item When $n=3$, we take $\frac{5}{2}\leqslant p\leqslant \frac{3}{(3-2s)_+}$ for $s\in[1,2]$, and $p>\max\{\frac{5}{2},s-1\}$ for $s\in(2,\infty)$.
		\item When $n=5$, we take $2\leqslant p\leqslant\frac{5}{5-2s}$ for $s\in[\frac{5}{4},2]$, and $2<p\leqslant \frac{9-2s}{(5-2s)_+}$ for $s\in(2,\frac{5}{2}]$, and $p\geqslant 2$ as well as $p>s-1$ for $s\in(\frac{5}{2},\infty)$.
		\item When $n=10$, we take $4<p\leqslant\frac{7-s}{(5-s)_+}$ for $s\in(\frac{13}{3},5]$, and $p>s-1$ for $s\in(5,\infty)$.
	\end{itemize}
\end{exam}
\begin{remark}
Let us consider the special case $n=2$ with $p=4$. Except the estimates stated in Theorem \ref{Thm-GESDS}, the solution also fulfills the additional estimates
\begin{align*}
\|u(t,\cdot)\|_{L^6}&\lesssim (1+t)^{-\frac{1}{6}}\|(u_0,u_1)\|_{\ml{A}_s},\\
\|u(t,\cdot)\|_{\dot{H}^1}&\lesssim (1+t)^{-\frac{1}{4}}\|(u_0,u_1)\|_{\ml{A}_s},
\end{align*}
and
\begin{align*}
\|u(t,\cdot)\|_{L^{\infty}}\lesssim (1+t)^{-\frac{1}{4}}\|(u_0,u_1)\|_{\ml{A}_s}\ \ \mbox{when}\ \ s\in[3,\infty).
\end{align*}
Thanks to these estimates, we may also get $L^q$ estimates of solutions by employing the Riesz-Thorin interpolation theorem.
\end{remark}
\begin{remark}
Although \cite[Theorems 5 and 6]{Su-Wang=2020} demonstrated global (in time) existence results for the nonlinear Cauchy problem \eqref{Eq-NonLinear-Dissipative-Boussinesq}, we cannot direct compare their ranges of $p$ with ours because the data spaces are quite different. To be specific, they considered $u_0\in \dot{H}^s\cap \dot{H}^{-1}_q$ and $u_1\in  \dot{H}^{s-2}\cap \dot{H}^{-2}_q$ with $q\in[1,2]$, namely, the homogeneous Sobolev spaces with additional Riesz potential spaces of negative order. But, in Theorem \ref{Thm-GESDS} we assume $u_0\in H^s\cap L^1$ and $H^{(s-4)_+}\cap L^1$, namely, inhomogeneous Sobolev spaces with additional $L^1$ integrability.
\end{remark}

To verify the optimality of the derived growth/decay estimates in Theorem \ref{Thm-GESDS} for large time, we have to estimate the lower bounds of global (in time) solution $u(t,\cdot)$ in the $L^2$ and $\dot{H}^s$ norms with $s\geqslant 1$. For this reason, we introduce the next theorem for a large time asymptotic profile of solutions, and obtain optimal lower bounds as a byproduct.

\begin{theorem}\label{Thm-Global-Profiles}
Let the exponent $p$, the regularity parameter $s$ and initial data $u_0,u_1$ satisfying the same hypotheses as those in Theorem \ref{Thm-GESDS}. Additionally, we assume $p>\frac{n+2}{n-1}$ for $n\geqslant 2$. Then, the global (in time) Sobolev solution introduced in Theorem \ref{Thm-GESDS} fulfills the next lower bound estimates:
\begin{align*}
	\|u(t,\cdot)\|_{L^2}&\gtrsim \ml{D}_n(t)|P_{u_1}|,\\
	\|u(t,\cdot)\|_{\dot{H}^s}&\gtrsim t^{-\frac{2s+n-2}{8}}|P_{u_1}|,
\end{align*}
for large time $t\gg1$, provided that $|P_{u_1}|\neq0$.
Furthermore, the following refined estimates hold:
\begin{align}\label{Bns}
\|u(t,\cdot)-u^{(1,p)}(t,\cdot)\|_{\dot{H}^s}=o(\ml{B}_{n,s}(t)):=\begin{cases}
	o(\ml{D}_n(t))&\mbox{when}\ \ s=0,\\
	o(t^{-\frac{2s+n-2}{8}})&\mbox{when}\ \ s\geqslant 1,
\end{cases}
\end{align}
for large time $t\gg1$, where the profile $u^{(1,p)}=u^{(1,p)}(t,x)$ is defined in \eqref{Nonlinear-Profile}.
\end{theorem}
\begin{remark}
Combining Theorem \ref{Thm-GESDS} and Theorem \ref{Thm-Global-Profiles}, provided that $(u_0,u_1)\in\ml{A}_s$ and $|P_{u_1}|\neq0$, when $n\geqslant 2$, the global (in time) Sobolev solution fulfills the optimal estimates \eqref{Optimal-Est-Nonlinear} for large time $t\gg1$. 
It is worth mentioning that the global (in time) solution grows logarithmically $\sqrt{\ln t}$ in the $L^2$ norm when $n=2$. Comparing with the logarithmic growth rate for the two-dimensional free wave equation in \cite{Ikehata=2023}, i.e. $\|u^{\mathrm{wave}}(t,\cdot)\|_{L^2}\simeq \sqrt{\ln t}$ when $n=2$, their growth rates  are exactly the same. This new phenomenon can be explained by the dispersion (wave effect) of the nonlinear dissipative Boussinesq equation \eqref{Eq-NonLinear-Dissipative-Boussinesq} playing the decisive role when $n=2$. Moreover, for higher dimensions $n\geqslant 3$, the dissipation (damping effect) exerts the crucial influence.
\end{remark}
\begin{remark}
One of our novelties is the discovery of asymptotic profile $u^{(1,p)}=u^{(1,p)}(t,x)$, which contains the diffusion waves function $\ml{G}_1(t,x)$ rewritten by
\begin{align*}
\ml{G}_1(t,x)=\ml{F}^{-1}_{\xi\to x}\left(\frac{1}{2i|\xi|}\left(\mathrm{e}^{i|\xi|t+\frac{i}{2}|\xi|^3t}-\mathrm{e}^{-i|\xi|t-\frac{i}{2}|\xi|^3t}\right)\mathrm{e}^{-\mu|\xi|^4t}\right).
\end{align*}
This result implies some influence of dispersion and dissipation on asymptotic behavior of solutions to the nonlinear dissipative Boussinesq equation \eqref{Eq-NonLinear-Dissipative-Boussinesq}.
\end{remark}

\section{Sharp large time asymptotic behavior for the linearized model}\label{Section-Linearized-Model}\setcounter{equation}{0}
$\ \ \ \ $ This section contributes to investigate some qualitative properties of solution, including optimal growth/decay estimates as well as optimal leading terms in the $L^2$ framework, to the linearized dissipative Boussinesq equation with vanishing right-hand side as follows:
\begin{align}\label{Eq-Linear-Dissipative-Boussinesq}
\begin{cases}
v_{tt}-\Delta v+\Delta^2v+2\mu\Delta^2v_t=0,&x\in\mb{R}^n,\ t>0,\\
v(0,x)=v_0(x),\ v_t(0,x)=v_1(x),&x\in\mb{R}^n,
\end{cases}
\end{align}
with $\mu\in(0,1)$. These results are the preliminaries for studying global (in time) Sobolev solution to the nonlinear dissipative Boussinesq equation \eqref{Eq-NonLinear-Dissipative-Boussinesq} in the forthcoming sections. Moreover, we discover the critical dimension $n=2$ for distinguishing decisive role between the dispersion (wave effect) and dissipation (damping effect) on large time behavior of solution. We refer to Table \ref{tab:table1}.
\subsection{Refined estimates in the Fourier space}\label{Sub-Section-Pretreatment-Linear}
$\ \ \ \ $ To begin with our analysis, let us apply the partial Fourier transform with respect to spatial variables to the linearized Cauchy problem \eqref{Eq-Linear-Dissipative-Boussinesq}. Then, it implies
\begin{align*}
\begin{cases}
\widehat{v}_{tt}+2\mu|\xi|^4\widehat{v}_t+(|\xi|^2+|\xi|^4)\widehat{v}=0,&\xi\in\mb{R}^n,\ t>0,\\
\widehat{v}(0,\xi)=\widehat{v}_0(\xi),\ \widehat{v}_t(0,\xi)=\widehat{v}_1(\xi),&\xi\in\mb{R}^n.
\end{cases}
\end{align*}
The corresponding characteristic equation $\lambda^2+2\mu|\xi|^4\lambda+|\xi|^2+|\xi|^4=0$ has the roots
\begin{align*}
\lambda_{\pm}=-\mu|\xi|^4\pm\sqrt{\mu^2|\xi|^8-|\xi|^2-|\xi|^4}.
\end{align*}
We notice the dissipative-dispersive structure from the roots leading to non-trivial growth/decay properties of solution.
Obviously, the real parts of $\lambda_{\pm}$ are negative for all $\xi\in\mb{R}^n$. Before applying WKB analysis, we take the next zones in the Fourier space:
\begin{align*}
	\ml{Z}_{\intt}(\varepsilon_0)&:=\{\xi\in\mb{R}^n:\ |\xi|\leqslant\varepsilon_0\ll1\},\\
	\ml{Z}_{\bdd}(\varepsilon_0,N_0)&:=\{\xi\in\mb{R}^n:\ \varepsilon_0\leqslant |\xi|\leqslant N_0\},\\  
	\ml{Z}_{\extt}(N_0)&:=\{\xi\in\mb{R}^n:\ |\xi|\geqslant N_0\gg1\}.
\end{align*}
The cut-off functions $\chi_{\intt}(\xi),\chi_{\bdd}(\xi),\chi_{\extt}(\xi)\in \mathcal{C}^{\infty}$ having supports in their corresponding zones $\ml{Z}_{\intt}(\varepsilon_0)$, $\ml{Z}_{\bdd}(\varepsilon_0/2,2N_0)$ and $\ml{Z}_{\extt}(N_0)$, respectively, fulfilling $\chi_{\bdd}(\xi)=1-\chi_{\intt}(\xi)-\chi_{\extt}(\xi)$. Let us carry out asymptotic expansions with suitable order.  When $\xi\in\ml{Z}_{\intt}(\varepsilon_0)$ with $\varepsilon_0\ll 1$, the characteristic roots can be expanded by
\begin{align}\label{Expan-small-01}
\lambda_{\pm}=-\mu|\xi|^4\pm i\left(|\xi|+\frac{1}{2}|\xi|^3-\frac{1}{8}|\xi|^5+O(|\xi|^7)\right)=:\lambda_{\mathrm{R}}\pm i\lambda_{\mathrm{I}}\ \ \mbox{when}\ \ |\xi|\leqslant\varepsilon_0,
\end{align}
in which the higher order expansion in the imaginary part contributes to the construction of second profiles of solution later.  For another, when $\xi\in\ml{Z}_{\extt}(N_0)$ with $N_0\gg1$,  one gets
\begin{align}\label{Expan-large-01}
\lambda_{\pm}=\begin{cases}
\displaystyle{-\frac{1}{2\mu}+O(|\xi|^{-2})}\\
\displaystyle{-2\mu|\xi|^4+O(1)}
\end{cases}\ \ \mbox{when}\ \ |\xi|\geqslant N_0.
\end{align}

Note that $\mathrm{Re}\,\lambda_{\pm}<0$ when $\xi\in\ml{Z}_{\bdd}(\varepsilon_0,N_0)$, which leads to an exponential decay of solution in this zone.
Concerning $\xi\in\ml{Z}_{\intt}(\varepsilon_0)\cup\ml{Z}_{\extt}(N_0)$, thanks to the pairwise distinct characteristic roots discussed in the above, it allows us to represent the solution $\widehat{v}=\widehat{v}(t,\xi)$ by
\begin{align*}
\widehat{v}=\widehat{K}_0\widehat{v}_0+\widehat{K}_1\widehat{v}_1:=\frac{\lambda_+\mathrm{e}^{\lambda_-t}-\lambda_-\mathrm{e}^{\lambda_+t}}{\lambda_+-\lambda_-}\widehat{v}_0+\frac{\mathrm{e}^{\lambda_+t}-\mathrm{e}^{\lambda_-t}}{\lambda_+-\lambda_-}\widehat{v}_1.
\end{align*}
Especially, some cancellations between two dispersive parts show
\begin{align}\label{Rep-Lin-01}
\widehat{v}=\left(\cos(\lambda_{\mathrm{I}}t)-\frac{\lambda_{\mathrm{R}}}{\lambda_{\mathrm{I}}}\sin(\lambda_{\mathrm{I}}t)\right)\mathrm{e}^{\lambda_{\mathrm{R}}t}\widehat{v}_0+\frac{\sin(\lambda_{\mathrm{I}}t)}{\lambda_{\mathrm{I}}}\mathrm{e}^{\lambda_{\mathrm{R}}t}\widehat{v}_1\ \ \mbox{when}\ \ \xi\in\ml{Z}_{\intt}(\varepsilon_0).
\end{align}
We later will discuss some combined influence from the dissipative part $\mathrm{e}^{\lambda_{\mathrm{R}}t}$, the oscillating part $\sin(\lambda_{\mathrm{I}}t)$ and the singular part $\lambda_{\mathrm{I}}^{-1}$ (for small frequencies).

Strongly motivated by the behavior for the characteristic roots $\lambda_{\pm}$, the large time asymptotic behavior of solution is determined by the small frequency part. We consequently state some error estimates via subtracting suitable approximation functions when $\xi\in\ml{Z}_{\intt}(\varepsilon_0)$.
\begin{prop}\label{Prop-Refined-Error-01}The kernels of $\widehat{v}$ for small frequencies fulfill the following error estimates:
\begin{align}
	\chi_{\intt}(\xi)\left|\widehat{K}_1-\widehat{\ml{G}}_1\right|&\lesssim \chi_{\intt}(\xi)\mathrm{e}^{-c|\xi|^4t},\label{Est-K1-01}\\
	\chi_{\intt}(\xi)\left|\widehat{K}_1-\widehat{\ml{G}}_{1}+\frac{t}{8}|\xi|^4\widehat{\ml{G}}_0\right|+\chi_{\intt}(\xi)\left|\widehat{K}_0-\widehat{\ml{G}}_0\right|&\lesssim \chi_{\intt}(\xi)|\xi|\mathrm{e}^{-c|\xi|^4t},\label{Est-K1-02-K0}
\end{align}
where the functions $\widehat{\ml{G}}_j=\widehat{\ml{G}}_j(t,|\xi|)$ are defined in \eqref{G0-G1-fun}.
\end{prop}
\begin{proof} Throughout this proof, we consider $\xi\in\ml{Z}_{\intt}(\varepsilon_0)$.
Let us start with the kernel $\widehat{K}_1$ by direct subtractions as follows:
\begin{align*}
\widehat{K}_1-\widehat{\ml{G}}_1&=\left(\widehat{K}_1-\frac{\sin(\lambda_{\mathrm{I}}t)}{|\xi|}\mathrm{e}^{\lambda_{\mathrm{R}}t}\right)+\left(\frac{\sin(\lambda_{\mathrm{I}}t)}{|\xi|}\mathrm{e}^{\lambda_{\mathrm{R}}t}-\widehat{\ml{G}}_1\right)\\
&=\frac{-|\xi|+O(|\xi|^3)}{2+O(|\xi|^2)}\sin(\lambda_{\mathrm{I}}t)\mathrm{e}^{-\mu|\xi|^4t}+\left(-\frac{1}{8}|\xi|^4+O(|\xi|^6)\right)t\cos(\eta_0t)\mathrm{e}^{-\mu|\xi|^4t},
\end{align*}
where we used the mean value theorem 
with $\eta_0\in(\lambda_{\mathrm{I}},|\xi|+\frac{1}{2}|\xi|^3)$. It leads to \eqref{Est-K1-01} immediately. To derive a further refined error estimate, we apply the second order expansion
\begin{align*}
\sin(\lambda_{\mathrm{I}}t)-\sin\left(|\xi|t+\frac{1}{2}|\xi|^3t\right)+\frac{t}{8}|\xi|^5\cos\left(|\xi|t+\frac{1}{2}|\xi|^3t\right)=O(|\xi|^7)t+O(|\xi|^{10})t^2
\end{align*}
implying
\begin{align*}
	\chi_{\intt}(\xi)\left|\widehat{K}_1-\widehat{\ml{G}}_{1}+\frac{t}{8}|\xi|^4\widehat{\ml{G}}_0\right|\lesssim\chi_{\intt}(\xi)|\xi|\mathrm{e}^{-c|\xi|^4t}+\chi_{\intt}(\xi)\left(O(|\xi|^6)t+O(|\xi|^9)t^2\right)\mathrm{e}^{-c|\xi|^4t}
\end{align*}
to arrive at the first part of our estimate \eqref{Est-K1-02-K0}. For another kernel $\widehat{K}_0$, we just repeat the same procedure as the above, namely,
\begin{align*}
\widehat{K}_0-\widehat{\ml{G}}_0&=\left[\cos(\lambda_{\mathrm{I}}t)-\cos\left(|\xi|t+\frac{1}{2}|\xi|^3t\right)\right]\mathrm{e}^{\lambda_{\mathrm{R}}t}-\frac{\lambda_{\mathrm{R}}}{\lambda_{\mathrm{I}}}\sin(\lambda_{\mathrm{I}}t)\mathrm{e}^{\lambda_{\mathrm{R}}t}\\
&=\left(\frac{1}{8}|\xi|^5+O(|\xi|^7)\right)t\sin(\eta_1t)\mathrm{e}^{-\mu|\xi|^4t}+\left(\mu|\xi|^3+O(|\xi|^5)\right)\sin(\lambda_{\mathrm{I}}t)\mathrm{e}^{-\mu|\xi|^4t}
\end{align*}
with $\eta_1\in(\lambda_{\mathrm{I}},|\xi|+\frac{1}{2}|\xi|^3)$, to derive the second part of our estimate \eqref{Est-K1-02-K0}.
\end{proof}

\subsection{Optimal growth/decay estimates for the linearized model}
$\ \ \ \ $Before deriving optimal estimates of solution, let us introduce the next useful lemma for estimating the Fourier multipliers, which is partly given by \cite[Lemmas 4.1 and 4.2]{Ikehata-Iyota=2018}.
\begin{lemma}\label{Lem-Optimal-Est}
The Fourier multipliers fulfill the following optimal estimates:
	\begin{align}
	\|\ml{G}_0(t,\cdot)\|_{\dot{H}^{s}}+\|\ml{G}_1(t,\cdot)\|_{\dot{H}^{s+1}}&\simeq t^{-\frac{2s+n}{8}},\label{Est-Optimal-01}\\
	\|\ml{G}_1(t,\cdot)\|_{L^2}&\simeq \ml{D}_n(t),\label{Est-Optimal-02}
	\end{align}
with $s\geqslant 0$ for large time $t\gg1$, where the time-dependent function $\ml{D}_n(t)$ is defined in \eqref{Dnt}. Note that the last two estimates still hold for the localized functions $\chi_{\intt}(D)\ml{G}_j(t,x)$ with $j=0,1$.
\end{lemma}
\begin{proof}Recalling \cite[Lemmas 4.1 and 4.2]{Ikehata-Iyota=2018}, the authors have derived \eqref{Est-Optimal-01} when $s=0$ and \eqref{Est-Optimal-02}, in which the oscillating functions are $\cos(|\xi|t)$ and $\sin(|\xi|t)$. Although the oscillating functions have been changed in $\ml{G}_j(t,x)$ with $j=0,1$, their results are still valid (see the next statement). For this reason, we just need to prove \eqref{Est-Optimal-01} for $s>0$.
First of all, let us apply the Plancherel formula and the polar coordinates to get the desired upper bound estimate
\begin{align*}
\|\ml{G}_0(t,\cdot)\|_{\dot{H}^{s}}^2&=\int_{\mb{R}^n}\left|\cos\left(|\xi|t+\frac{1}{2}|\xi|^3t\right)\right|^2|\xi|^{2s}\mathrm{e}^{-2\mu|\xi|^4t}\mathrm{d}\xi\lesssim\int_0^{\infty}r^{2s+n-1}\mathrm{e}^{-2\mu r^4t}\mathrm{d}r\lesssim t^{-\frac{2s+n}{4}}.
\end{align*}
Again, from the polar coordinates and $2\cos^2y=1+\cos2y$, we may decompose it into two integrals
\begin{align*}
\|\ml{G}_0(t,\cdot)\|_{\dot{H}^{s}}^2&=\int_{|\omega|=1}\int_0^{\infty}\left|\cos\left(rt+\frac{1}{2}r^3t\right)\right|^2r^{2s+n-1}\mathrm{e}^{-2\mu r^4t}\mathrm{d}r\mathrm{d}\sigma_{\omega}\\
&=C_nt^{-\frac{2s+n}{4}}\left(\int_0^{\infty}\eta^{2s+n-1}\mathrm{e}^{-2\mu\eta^4}\mathrm{d}\eta+\int_0^{\infty}\cos\left(2\eta t^{\frac{3}{4}}+\eta^3t^{\frac{1}{4}}\right)\eta^{2s+n-1}\mathrm{e}^{-2\mu\eta^4}\mathrm{d}\eta\right)
\end{align*}
with a positive constant $C_n$ depending on the dimension $n$. Due to the fact that $\eta^{2s+n-1}\mathrm{e}^{-2\mu\eta^4}\in L^1$ for $2s+n-1\geqslant 0$, according to the Riemann-Lebesgue theorem for large time $t\gg1$, we may get
\begin{align*}
	\left|\int_0^{\infty}\cos\left(2\eta t^{\frac{3}{4}}+\eta^3t^{\frac{1}{4}}\right)\eta^{2s+n-1}\mathrm{e}^{-2\mu\eta^4}\mathrm{d}\eta\right|\leqslant \frac{1}{2}\int_0^{\infty}\eta^{2s+n-1}\mathrm{e}^{-2\mu\eta^4}\mathrm{d}\eta
\end{align*}
so that
 $\|\ml{G}_0(t,\cdot)\|_{\dot{H}^s}^2\gtrsim t^{-\frac{2s+n}{4}}$. Summarizing the previous computation, we finish the first part of optimal estimates \eqref{Est-Optimal-01}. For the second part of it, we may complete the estimate by following the same method as the above.
%
\end{proof}

\begin{theorem}\label{Thm-Linear-Opt-Est}
Suppose that initial data $v_0\in H^{s+1}\cap L^1$ and $v_1\in H^{(s-3)_+}\cap L^1$ with $s\geqslant 0$. Then, the solution to the linearized dissipative Boussinesq equation \eqref{Eq-Linear-Dissipative-Boussinesq} fulfills the following optimal estimates:
\begin{align}
\ml{D}_n(t)|P_{v_1}|\lesssim\|v(t,\cdot)\|_{L^2}\lesssim\ml{D}_n(t)\|(v_0,v_1)\|_{(L^2\cap L^1)^2}\label{Est-Lin-Solution}
\end{align}
and
\begin{align}
t^{-\frac{2s+n}{8}}|P_{v_1}|\lesssim\|v(t,\cdot)\|_{\dot{H}^{s+1}}\lesssim t^{-\frac{2s+n}{8}}\|(v_0,v_1)\|_{(H^{s+1}\cap L^1)\times(H^{(s-3)_+}\cap L^1)}\label{Est-Lin-Derivative}
\end{align}
for large time $t\gg1$, provided that $|P_{v_1}|\neq0$.
\end{theorem}
\begin{proof}
Concerning the solution itself, according to the representation \eqref{Rep-Lin-01} supplemented by the asymptotic expansion \eqref{Expan-small-01}, or Proposition \ref{Prop-Refined-Error-01} together with the triangle inequality, we are able to derive
\begin{align*}
\|\chi_{\intt}(\xi)\widehat{v}(t,\xi)\|_{L^2}&\lesssim\sum\limits_{k=0,1} \left\|\chi_{\intt}(\xi)\left(|\xi|^{1-k}\mathrm{e}^{-c|\xi|^4t}+|\widehat{\ml{G}}_k(t,|\xi|)|\right)|\widehat{v}_k(\xi)|\right\|_{L^2}\\
&\lesssim t^{-\frac{n}{8}}\|v_0\|_{L^1}+\ml{D}_n(t)\|v_1\|_{L^1},
\end{align*}
where we used the Hausdorff-Young inequality and Lemma \ref{Lem-Optimal-Est} in the last line. Similarly, for the bounded and large frequency parts, from the asymptotic expansion \eqref{Expan-large-01}, it follows
\begin{align*}
\left\|\big(\chi_{\bdd}(\xi)+\chi_{\extt}(\xi)\big)\widehat{v}(t,\xi)\right\|_{L^2}\lesssim\mathrm{e}^{-ct}\left\||\widehat{v}_0(\xi)|+\langle\xi\rangle^{-4}|\widehat{v}_1(\xi)|\right\|_{L^2}\lesssim\mathrm{e}^{-ct}\|(v_0,v_1)\|_{(L^2)^2}.
\end{align*}
Summarizing the last two estimates and applying the Plancherel theorem from the Fourier space to the physical space, we claim the upper bound estimate in \eqref{Est-Lin-Solution} immediately. For the higher order derivatives, by the analogous way, we have
\begin{align*}
	\left\||\xi|^{s+1}\widehat{v}(t,\xi)\right\|_{L^2}&\lesssim\sum\limits_{k=0,1}\left\|\chi_{\intt}(\xi)|\xi|^{s+1-k}\mathrm{e}^{-c|\xi|^4t}|\widehat{v}_k(\xi)|\right\|_{L^2}+\mathrm{e}^{-ct}\left\|\langle\xi\rangle^{s+1}|\widehat{v}_0(\xi)|+\langle\xi\rangle^{s-3}|\widehat{v}_1(\xi)|\right\|_{L^2}\\
	&\lesssim t^{-\frac{2s+2+n}{8}}\|v_0\|_{H^{s+1}\cap L^1}+t^{-\frac{2s+n}{8}}\|v_1\|_{H^{(s-3)_+}\cap L^1},
\end{align*}
which completes the proof of upper bound estimates by the Plancherel theorem.

To guarantee the sharpness of the above estimates, we now turn to their lower bounds. With the aid of \eqref{Est-K1-01}, one gets
\begin{align}\label{Est-optimal-proc-01}
\|v(t,\cdot)-\ml{G}_1(t,|D|)v_1(\cdot)\|_{L^2}=\left\|\widehat{v}(t,\xi)-\widehat{\ml{G}}_1(t,|\xi|)\widehat{v}_1(\xi)\right\|_{L^2}\lesssim t^{-\frac{n}{8}}\|(v_0,v_1)\|_{(L^2\cap L^1)^2}.
\end{align}
For another, associated with the mean value theorem
\begin{align}\label{Taylor-first}
|\ml{G}_1(t,x-y)-\ml{G}_1(t,x)|\lesssim|y|\,|\nabla\ml{G}_1(t,x-\eta_2y)|
\end{align}
with $\eta_2\in(0,1)$, we may separate the integral into two parts such that
\begin{align}
&\|\ml{G}_1(t,|D|)v_1(\cdot)-\ml{G}_1(t,\cdot)P_{v_1}\|_{L^2}\notag\\
&\quad\lesssim\left\|\int_{|y|\leqslant t^{\frac{1}{16}}}\big(\ml{G}_1(t,\cdot-y)-\ml{G}_1(t,\cdot)\big)v_1(y)\mathrm{d}y\right\|_{L^2}+\left\|\int_{|y|\geqslant t^{\frac{1}{16}}}\big(|\ml{G}_1(t,\cdot-y)|+|\ml{G}_1(t,\cdot)|\big)|v_1(y)|\mathrm{d}y\right\|_{L^2}\notag\\
&\quad\lesssim t^{\frac{1}{16}}\|\,|\xi|\widehat{\ml{G}}_1(t,|\xi|)\|_{L^2}\|v_1\|_{L^1}+\|\widehat{\ml{G}}_1(t,|\xi|)\|_{L^2}\|v_1\|_{L^1(|x|\geqslant t^{\frac{1}{16}})}\notag\\
&\quad\lesssim t^{\frac{1}{16}-\frac{n}{8}}\|v_1\|_{L^1}+o(\ml{D}_n(t))\label{Est-optimal-proc-02}
\end{align}
for large time $t\gg1$, thanks to our assumption on the $L^1$ integrability of $v_1$. Hence, the triangle inequality for \eqref{Est-optimal-proc-01} as well as \eqref{Est-optimal-proc-02} shows
\begin{align}\label{Est-new}
\|v(t,\cdot)-\ml{G}_1(t,\cdot)P_{v_1}\|_{L^2}\lesssim t^{-\frac{n}{8}}\|(v_0,v_1)\|_{(L^2\cap L^1)^2}+t^{\frac{1}{16}-\frac{n}{8}}\|v_1\|_{L^1}+o(\ml{D}_n(t)).
\end{align}
Eventually, let us employ the Minkowski inequality and combine the last estimate with Lemma \ref{Lem-Optimal-Est} to arrive at
\begin{align*}
\|v(t,\cdot)\|_{L^2}&\gtrsim\|\ml{G}_1(t,\cdot)\|_{L^2}|P_{v_1}|-\|v(t,\cdot)-\ml{G}_1(t,\cdot)P_{v_1}\|_{L^2}\\
&\gtrsim\ml{D}_n(t)|P_{v_1}|-t^{-\frac{n}{8}}\|(v_0,v_1)\|_{(L^2\cap L^1)^2}-t^{\frac{1}{16}-\frac{n}{8}}\|v_1\|_{L^1}-o(\ml{D}_n(t))\\
&\gtrsim \ml{D}_n(t)|P_{v_1}|
\end{align*}
for large time $t\gg1$. The derivation of sharp lower bound estimate in \eqref{Est-Lin-Solution} is finished. By the same manner, we may derive the lower bound estimate in \eqref{Est-Lin-Derivative} without any additional difficulty. Thus, our proof is completed.
\end{proof}
\begin{remark}
	For the linearized Boussinesq equation with structural damping $(-\Delta)^{\theta}v_t$ carrying $\theta\in[0,1]$, the authors of \cite{Ikehata-Soga=2015} obtained optimal growth/decay estimates with weighted $L^1$ data. Here, we can relax such assumption to the $L^1$ integrability only because of \eqref{Est-optimal-proc-02}. 
\end{remark}
\begin{remark}
The previous paper \cite[Theorem 2]{Su-Wang=2020} derived optimal decay estimates for the Cauchy problem \eqref{Eq-Linear-Dissipative-Boussinesq} with initial data belonging to Riesz potential spaces of negative order and suitable weighted $L^1$ space, in which the operator $(-\Delta)^{-1}$ plays an indispensable role in the assumptions of initial data. Different from their work, we consider the usual Sobolev spaces with additional $L^1$ integrability for initial data, and derive optimal growth ($n=1,2$) and decay ($n\geqslant 3$) estimates of the solution itself in Theorem \ref{Thm-Linear-Opt-Est}. This growth property in lower dimensions is caused by dispersion (wave effect) in the dissipative Boussinesq equation (see Remark \ref{Rem-3.3} later).
\end{remark}
\begin{remark}
It is worth noting that the growth rates for the free wave equation (see \cite[Theorems 1 and 2]{Ikehata=2023}) and the linearized dissipative Boussinesq equation \eqref{Eq-Linear-Dissipative-Boussinesq} are exactly the same in  lower dimensions $n=1,2$, but the solution decays polynomially with the help of the dissipation $\Delta^2v_t$ in higher dimensions $n\geqslant 3$.  Let us summary them by Table \ref{tab:table1}.
\renewcommand\arraystretch{1.4}
\begin{table}[h!]
	\begin{center}
		\caption{Influence from dispersion (wave effect) and dissipation (damping effect)}
		\medskip
		\label{tab:table1}
		\begin{tabular}{cccc} 
			\toprule
			Dimension & $n=1$ & $n=2$ & $n\geqslant3$\\
			\midrule
			Dispersion property (wave effect) & $\sqrt{t}$ & $\sqrt{\ln t}$ & -- \\
			Dissipation property (damping effect) & $t^{-\frac{1}{8}}$& $t^{-\frac{2}{8}}$& $t^{-\frac{n}{8}}$\\  
			Dissipative Boussinesq equation property & $\sqrt{t}$ & $\sqrt{\ln t}$ & $t^{-\frac{n-2}{8}}=t^{\frac{1}{4}}\cdot t^{-\frac{n}{8}}$\\
			\hline 
			Crucial influence & Dispersion & Dispersion & Dispersion \& Dissipation  \\
			\bottomrule
			\multicolumn{4}{l}{\emph{$*$ The terminology ``property'' specializes the time-dependent coefficient in the $L^2$ estimate of solution.}}
		\end{tabular}
	\end{center}
\end{table}

\noindent In other words, the dispersion plays the decisive role when $n=1,2$, but the dissipation exerts the crucial influence when $n\geqslant 3$, which is one of our new discoveries in the linearized dissipative Boussinesq equation \eqref{Eq-Linear-Dissipative-Boussinesq}.
\end{remark}

Finally, the pointwise estimate with $s\geqslant0$ holds
\begin{align}\label{Est-Small-Point}
\chi_{\intt}(\xi)|\xi|^{s+2}|\widehat{v}|\lesssim\chi_{\intt}(\xi)\mathrm{e}^{-c|\xi|^4t}(|\xi|^{s+2}|\widehat{v}_0|+|\xi|^{s+1}|\widehat{v}_1|).
\end{align}
By applying Theorem \ref{Thm-Linear-Opt-Est} when $t\geqslant t_0\gg1$ and the bounded estimate for $\|v(t,\cdot)\|_{\dot{H}^{s+2}}$ with $s\geqslant0$ when $t\leqslant t_0$, we can easily obtain the next $(L^2\cap L^1)-L^2$ type estimates. The other $L^2-L^2$ type estimates can be derived by employing \eqref{Est-Small-Point}.
\begin{coro}\label{Coro-Linear-Est}
Suppose that initial data $v_0\in H^{s+2}\cap L^1$ and $v_1\in H^{(s-2)_+}\cap L^1$ with $s\geqslant 0$. Then, the solution fulfills the following  $(L^2\cap L^1)-L^2$ type estimates:
\begin{align*}
\|v(t,\cdot)\|_{\dot{H}^{s+2}}\lesssim (1+t)^{-\frac{2(s+1)+n}{8}}\|(v_0,v_1)\|_{(\dot{H}^{s+2}\cap L^1)\times(\dot{H}^{(s-2)_+}\cap L^1)},
\end{align*}
and the following $L^2- L^2$ type estimates:
\begin{align*}
	\|v(t,\cdot)\|_{\dot{H}^{s+2}}\lesssim \|v_0\|_{\dot{H}^{s+2}}+\begin{cases}
		(1+t)^{-\frac{s+1}{4}}\|v_1\|_{L^2}&\mbox{when}\ \ s\in[0,2],\\
		(1+t)^{-\frac{3}{4}}\|v_1\|_{\dot{H}^{s-2}}&\mbox{when}\ \ s\in(2,\infty).
	\end{cases}
\end{align*}
\end{coro}

To end this part, we propose some estimates of solution in the some norms when $n=2$. They will contribute to the global (in time) existence result for the nonlinear problem \eqref{Eq-NonLinear-Dissipative-Boussinesq} in the special case $n=2$ with $p=4$.
\begin{coro}\label{Coro-Linear-L6}
	Let $n=2$. Suppose that initial data $v_0\in H^{1+\epsilon_1}\cap L^1$ and $v_1\in L^2\cap L^1$ with a sufficiently small constant $\epsilon_1>0$. Then, the solution fulfills the following estimates:
	\begin{align*}
		\|v(t,\cdot)\|_{L^{\infty}}&\lesssim (1+t)^{-\frac{1}{4}}\|(v_0,v_1)\|_{(H^{1+\epsilon_1}\cap L^1)\times (L^2\cap L^1)},\\
		\|v(t,\cdot)\|_{L^6}&\lesssim (1+t)^{-\frac{1}{6}}\|(v_0,v_1)\|_{(H^{\frac{2}{3}+\epsilon_1}\cap L^1)\times (L^2\cap L^1)}.
	\end{align*}
\end{coro}
\begin{proof}
First of all, by applying the Hausdorff-Young inequality associated with the derived pointwise estimates in the Fourier space, we may obtain
\begin{align*}
\|v(t,\cdot)\|_{L^{\infty}}
&\lesssim\sum\limits_{k=0,1}\big\|\chi_{\intt}(\xi)|\xi|^{-k}\mathrm{e}^{-c|\xi|^4t}|\widehat{v}_k(\xi)|\big\|_{L^1}+\mathrm{e}^{-ct}\big\|\big(1-\chi_{\intt}(\xi)\big)\big(|\widehat{v}_0(\xi)|+\langle\xi\rangle^{-4}|\widehat{v}_1(\xi)|\big)\big\|_{L^1}\\
&\lesssim\sum\limits_{k=0,1}\big\|\chi_{\intt}(\xi)|\xi|^{-k}\mathrm{e}^{-c|\xi|^4t}\big\|_{L^1}\|\widehat{v}_k\|_{L^{\infty}}+\mathrm{e}^{-ct}\big\|\big(1-\chi_{\intt}(\xi)\big)\langle\xi\rangle^{-(1+\epsilon_1)}\big\|_{L^2}\|(v_0,v_1)\|_{H^{1+\epsilon_1}\times L^2}
\end{align*}
with $\epsilon_1>0$. According to the facts that
\begin{align*}
\sum\limits_{k=0,1}\int_0^{\varepsilon_0}r^{-k+1}\mathrm{e}^{-cr^4t}\mathrm{d}r\lesssim (1+t)^{-\frac{1}{4}}\ \ \mbox{and}\ \ \int_{\varepsilon_0}^{\infty}\langle r\rangle^{-2(1+\epsilon_1)}r\mathrm{d}r<\infty,
\end{align*}
one concludes
\begin{align*}
	\|v(t,\cdot)\|_{L^{\infty}}\lesssim (1+t)^{-\frac{1}{4}}\|(v_0,v_1)\|_{(L^1)^2}+\mathrm{e}^{-ct}\|(v_0,v_1)\|_{H^{1+\epsilon_1}\times L^2}.
\end{align*}
By the same way as the above deduction, we also can derive
\begin{align*}
\|v(t,\cdot)\|_{L^6}&\lesssim\sum\limits_{k=0,1}\big\|\chi_{\intt}(\xi)|\xi|^{-k}\mathrm{e}^{-c|\xi|^4t}|\widehat{v}_k(\xi)|\big\|_{L^{\frac{6}{5}}}\|v_k\|_{L^1}\\
&\quad+\mathrm{e}^{-ct}\big\|\big(1-\chi_{\intt}(\xi)\big)\langle\xi\rangle^{-(\frac{2}{3}+\epsilon_1)}\big\|_{L^3}\|(v_0,v_1)\|_{H^{\frac{2}{3}+\epsilon_1}\times L^2}\\
&\lesssim (1+t)^{-\frac{1}{6}}\|(v_0,v_1)\|_{(L^1)^2}+\mathrm{e}^{-ct}\|(v_0,v_1)\|_{H^{\frac{2}{3}+\epsilon_1}\times L^2}
\end{align*}
with $\epsilon_1>0$. Our proof is complete.
\end{proof}
\begin{coro}\label{Coro-New}
Let $n=2$. Suppose that the source term $g=g(x)$ fulfills $|D|g\in L^1$ and $g\in\dot{H}^{(s-2)_+}$. Then, the following estimates hold:
\begin{align*}
\|\,|D|K_1(t,\cdot)\ast_{(x)}|D|g(\cdot)\|_{L^2}&\lesssim (1+t)^{-\frac{1}{4}}\|\,|D|g\|_{L^1}+\mathrm{e}^{-ct}\|g\|_{L^2},\\
\|\,|D|K_1(t,\cdot)\ast_{(x)}|D|g(\cdot)\|_{L^6}&\lesssim (1+t)^{-\frac{5}{12}}\|\,|D|g\|_{L^1}+\mathrm{e}^{-ct}\|g\|_{L^2},\\
\|\,|D|K_1(t,\cdot)\ast_{(x)}|D|g(\cdot)\|_{L^{\infty}}&\lesssim (1+t)^{-\frac{1}{2}}\|\,|D|g\|_{L^1}+\mathrm{e}^{-ct}\|g\|_{L^2},\\
\|\,|D|K_1(t,\cdot)\ast_{(x)}|D|g(\cdot)\|_{\dot{H}^s}&\lesssim (1+t)^{-\frac{s+1}{4}}\|\,|D|g\|_{L^1}+\mathrm{e}^{-ct}\|g\|_{\dot{H}^{(s-2)_+}}.
\end{align*}
\end{coro}
\begin{proof}
Repeating the same procedures as those in the proofs of Theorem \ref{Thm-Linear-Opt-Est} and Corollary \ref{Coro-Linear-L6}, we may directly deduce all estimates.
\end{proof}

\subsection{Optimal leading terms for the linearized model}
$\ \ \ \ $Let us introduce the large time profiles $v^{(k,p)}=v^{(k,p)}(t,x)$ with $k=1,2$ for the linearized dissipative Boussinesq equation \eqref{Eq-Linear-Dissipative-Boussinesq} in the formulas \eqref{Profile-v1p} and \eqref{Profile-v2p}. We now recall the weighted $L^1$ space via
\begin{align*}
	L^{1,1}:=\left\{f\in L^1 \ :\ \|f\|_{L^{1,1}}:=\int_{\mb{R}^n}(1+|x|)|f(x)|\mathrm{d}x<\infty \right\}.
\end{align*}
We then conclude the following optimal decay estimates for the error term so that one may claim the function $v^{(1,p)}$ to be the optimal leading term of solution.
\begin{theorem}\label{Thm-Linear-Leading-Term}
Suppose that initial data $v_0\in H^s\cap L^1$ and $v_1\in H^s\cap L^{1,1}$ with $s\geqslant 0$. Then, the solution to the linearized dissipative Boussinesq equation \eqref{Eq-Linear-Dissipative-Boussinesq} fulfills the following optimal error estimate:
\begin{align}\label{Est-Lin-Lead-01}
	t^{-\frac{2s+n}{8}}\mb{A}_{\lin}\lesssim\|v(t,\cdot)-v^{(1,p)}(t,\cdot)\|_{\dot{H}^s}\lesssim t^{-\frac{2s+n}{8}}\|(v_0,v_1)\|_{(H^s\cap L^1)\times(H^{(s-1)_+}\cap L^{1,1})}
\end{align}
for large time $t\gg1$, provided that $0\neq\mb{A}_{\lin}:=|P_{v_1}|+|P_{v_c}|+|M_{v_1}|$ carrying the combined data $v_c(x):=v_0(x)-\frac{2s+n}{64\mu}v_1(x)$. Moreover, the further error estimate holds
\begin{align}\label{Est-Lin-Lead-02}
\|v(t,\cdot)-v^{(1,p)}(t,\cdot)-v^{(2,p)}(t,\cdot)\|_{\dot{H}^s}=o(t^{-\frac{2s+n}{8}})
\end{align}
for large time $t\gg1$, where the right-hand side depends on $\|v_0\|_{H^s\cap L^1}$ and $\|v_1\|_{H^s\cap L^{1,1}}$.
\end{theorem}
\begin{proof}
Our starting point is to prove the estimate \eqref{Est-Lin-Lead-02}. We carry out the decomposition
\begin{align*}
v(t,x)-v^{(1,p)}(t,x)-v^{(2,p)}(t,x)=\sum\limits_{k=1,\dots,4}\ml{E}_k(t,x),
\end{align*}
whose components are chosen by
\begin{align*}
\ml{E}_1(t,x)&:=v(t,x)-\ml{G}_1(t,|D|)v_1(x)+\frac{t}{8}\Delta^2\ml{G}_0(t,|D|)v_1(x)-\ml{G}_0(t,|D|)v_0(x),\\
\ml{E}_2(t,x)&:=\ml{G}_1(t,|D|)v_1(x)-\ml{G}_1(t,x)P_{v_1}+\nabla\ml{G}_1(t,x)\circ M_{v_1},\\
\ml{E}_3(t,x)&:=-\frac{t}{8}\left(\Delta^2\ml{G}_0(t,|D|)v_1(x)-\Delta^2\ml{G}_0(t,x)P_{v_1}\right),\\
\ml{E}_4(t,x)&:=\ml{G}_0(t,|D|)v_0(x)-\ml{G}_0(t,x)P_{v_0}.
\end{align*}
Recalling the pointwise estimate \eqref{Est-K1-02-K0}, we find
\begin{align*}
\|\chi_{\intt}(\xi)|\xi|^s\widehat{\ml{E}}_1(t,\xi)\|_{L^2}\lesssim\left\|\chi_{\intt}(\xi)|\xi|^{s+1}\mathrm{e}^{-c|\xi|^4t}\right\|_{L^2}\|(v_0,v_1)\|_{(L^1)^2}\lesssim t^{-\frac{2s+2+n}{8}}\|(v_0,v_1)\|_{(L^1)^2}.
\end{align*}
Concerning bounded and large frequencies, it holds
\begin{align*}
\left\|\big(\chi_{\bdd}(\xi)+\chi_{\extt}(\xi)\big)|\xi|^s\widehat{\ml{E}}_1(t,\xi)\right\|_{L^2}\lesssim\mathrm{e}^{-ct}\|(v_0,v_1)\|_{(H^s)^2}
\end{align*}
so that
\begin{align*}
\|\ml{E}_1(t,\cdot)\|_{\dot{H}^s}\lesssim t^{-\frac{2s+2+n}{8}}\|(v_0,v_1)\|_{(H^s\cap L^1)^2}.
\end{align*}
For the third and fourth error terms, analogously to \eqref{Est-optimal-proc-02}, we are able to deduce
\begin{align*}
\|\ml{E}_3(t,\cdot)\|_{\dot{H}^s}+\|\ml{E}_4(t,\cdot)\|_{\dot{H}^s}=o(t^{-\frac{2s+n}{8}})
\end{align*}
for large time $t\gg1$, since the assumption $v_0,v_1\in L^1$. In order to treat the final error term $\ml{E}_2(t,x)$, we consider Taylor's expansion
\begin{align}\label{Taylor-second}
|\ml{G}_1(t,x-y)-\ml{G}_1(t,x)+y\circ\nabla\ml{G}_1(t,x)|\lesssim |y|^2|\nabla^2\ml{G}_1(t,x-\eta_3y)|
\end{align}
with $\eta_3\in(0,1)$. Moreover, $\ml{E}_2(t,x)$ can be separated into the sum of next terms:
\begin{align*}
\ml{E}_{2,1}(t,x)&:=\int_{|y|\leqslant t^{\frac{1}{32}}}\big(\ml{G}_1(t,x-y)-\ml{G}_1(t,x)+y\circ\nabla\ml{G}_1(t,x)\big)v_1(y)\mathrm{d}y,\\
\ml{E}_{2,2}(t,x)&:=\int_{|y|\geqslant t^{\frac{1}{32}}}\big(\ml{G}_1(t,x-y)-\ml{G}_1(t,x)\big)v_1(y)\mathrm{d}y+\int_{|y|\geqslant t^{\frac{1}{32}}}y\circ\nabla \ml{G}_1(t,x)v_1(y)\mathrm{d}y.
\end{align*}
By using \eqref{Taylor-second},  we obtain
\begin{align*}
	\|\ml{E}_{2,1}(t,\cdot)\|_{\dot{H}^s}\lesssim t^{\frac{1}{16}}\|\,|\xi|^{s+2}\widehat{\ml{G}}_1(t,|\xi|)\|_{L^2}\|v_1\|_{L^1}\lesssim t^{-\frac{3}{16}-\frac{2s+n}{8}}\|v_1\|_{L^1}.
\end{align*}
Moreover, the inequality \eqref{Taylor-first} implies
\begin{align*}
\|\ml{E}_{2,2}(t,\cdot)\|_{\dot{H}^s}&\lesssim \|\,|\xi|^{s+1}\widehat{\ml{G}}_1(t,|\xi|)\|_{L^2}\|\,|x|v_1\|_{L^1(|x|\geqslant t^{\frac{1}{32}})}=o(t^{-\frac{2s+n}{8}})
\end{align*}
for large time $t\gg1$, thanks to the additional weighted $L^1$ assumption $v_1\in L^{1,1}$. It gives
\begin{align*}
\sum\limits_{k=1,\dots,4}\|\ml{E}_k(t,\cdot)\|_{\dot{H}^s}\lesssim t^{-\frac{1}{4}-\frac{2s+n}{8}}\|(v_0,v_1)\|_{(H^s\cap L^1)^2}+t^{-\frac{3}{16}-\frac{2s+n}{8}}\|v_1\|_{L^1}+o(t^{-\frac{2s+n}{8}}),
\end{align*}
which leads to our desire error estimate \eqref{Est-Lin-Lead-02}.

Indeed, due to the hypothesis $v_1\in L^{1,1}$ and \cite[Lemma 2.2]{Ikehata=2014}, we may deduce
\begin{align*}
|\widehat{v}_1(\xi)-P_{v_1}|\lesssim|\xi|\,\|v_1\|_{L^{1,1}}
\end{align*}
so that the estimate \eqref{Est-optimal-proc-02} is improved by
\begin{align*}
\left\|\chi_{\intt}(\xi)|\xi|^s\widehat{\ml{G}}_1(t,|\xi|)\big(\widehat{v}_1(\xi)-P_{v_1}\big)\right\|_{L^2}\lesssim\left\|\chi_{\intt}(\xi)|\xi|^s\mathrm{e}^{-c|\xi|^4t}\right\|_{L^2}\|v_1\|_{L^{1,1}}\lesssim t^{-\frac{2s+n}{8}}\|v_1\|_{L^{1,1}}.
\end{align*}
According to the last line and the next estimate (deduced from Proposition \ref{Prop-Refined-Error-01}) associated with the triangle inequality:
\begin{align*}
\left\|\chi_{\intt}(\xi)|\xi|^s\big(\widehat{v}(t,\xi)-\widehat{\ml{G}}_1(t,|\xi|)\widehat{v}_1(\xi)\big)\right\|_{L^2}\lesssim t^{-\frac{2s+n}{8}}\|(v_0,v_1)\|_{( L^1)^2},
\end{align*}
we complete the proof of the upper bound estimate in \eqref{Est-Lin-Lead-01} immediately.

In order to derive the optimal lower bound estimate for the error term in \eqref{Est-Lin-Lead-01}, the crucial point is to estimate the second profile $v^{(2,p)}(t,\cdot)$ in the $\dot{H}^s$ norm from the below. The partial Fourier transform for the profile is given by
\begin{align*}
	\widehat{v}^{(2,p)}(t,\xi)=\widehat{\ml{G}}_0(t,|\xi|)P_{v_0}-i\widehat{\ml{G}}_1(t,|\xi|)\xi\circ M_{v_1}-\frac{t}{8}|\xi|^4\widehat{\ml{G}}_0(t,|\xi|)P_{v_1}.
\end{align*}
As a consequence, one notices
\begin{align*}
\|\,|\xi|^s\widehat{v}^{(2,p)}(t,\xi)\|_{L^2}^2&=\int_{\mb{R}^n}|\xi|^{2s}|\widehat{\ml{G}}_1(t,|\xi|)|^2|\xi\circ M_{v_1}|^2\mathrm{d}\xi+\int_{\mb{R}^n}|\xi|^{2s}\left|P_{v_0}-\frac{t}{8}|\xi|^4P_{v_1}\right|^2|\widehat{\ml{G}}_0(t,|\xi|)|^2\mathrm{d}\xi\\
&=:A_1(t)+A_2(t).
\end{align*}
By applying the polar coordinates, we arrive at
\begin{align}
A_1(t)&=\int_0^{\infty}r^{2s+n+1}|\widehat{\ml{G}}_1(t,r)|^2\mathrm{d}r\int_{|\omega|=1}|\omega\circ M_{v_1}|^2\mathrm{d}\sigma_{\omega}\notag\\
&\gtrsim\|\chi_{\intt}(\xi)|\xi|^{s+1}\widehat{\ml{G}}_1(t,|\xi|)\|_{L^2}^2|M_{v_1}|^2\notag\\
&\gtrsim t^{-\frac{2s+n}{4}}|M_{v_1}|^2,\label{Est-A1}
\end{align}
where the optimal estimate \eqref{Est-Optimal-01} was used. For another, by setting $\eta=rt^{\frac{1}{4}}$, we may separate $A_2(t)$ into three parts via
\begin{align}
A_2(t)&=C_n\int_0^{\infty}r^{2s+n-1}\left|P_{v_0}-\frac{t}{8}r^4P_{v_1}\right|^2|\widehat{\ml{G}}_0(t,r)|^2\mathrm{d}r\notag\\
&=C_n\int_0^{\infty}r^{2s+n-1}\left(|P_{v_0}|^2-\frac{t}{4}r^4P_{v_0}P_{v_1}+\frac{t^2}{64}r^8|P_{v_1}|^2\right)\left|\cos\left(rt+\frac{1}{2}r^3t\right)\right|^2\mathrm{e}^{-2\mu r^4t}\mathrm{d}r\notag\\
&=C_n t^{-\frac{2s+n}{4}}\big(A_{2,0}(t)+A_{2,1}(t)+A_{2,2}(t)\big)\label{Est-A2}
\end{align}
with a positive constant $C_n$, where we took
\begin{align*}
A_{2,0}(t)=|P_{v_0}|^2\widetilde{A}(t;0),\ \ 
A_{2,1}(t)=-\frac{1}{4}P_{v_0}P_{v_1}\widetilde{A}(t;1),\ \ 
A_{2,2}(t)=\frac{1}{64}|P_{v_1}|^2\widetilde{A}(t;2),
\end{align*}
carrying the time-dependent function with $m=0,1,2$ as follows:
\begin{align*}
\widetilde{A}(t;m):=\int_0^{\infty}\eta^{2s+n-1+4m}\left|\cos\left(\eta t^{\frac{3}{4}}+\frac{1}{2}\eta^3t^{\frac{1}{4}}\right)\right|^2\mathrm{e}^{-2\mu\eta^4}\mathrm{d}\eta.
\end{align*}
Strongly motivated by the proof of Lemma \ref{Lem-Optimal-Est}, the Riemann-Lebesgue theorem yields
\begin{align*}
\widetilde{A}(t;m)&=\frac{1}{2}\int_0^{\infty}\eta^{2s+n-1+4m}\mathrm{e}^{-2\mu\eta^4}\mathrm{d}\eta+\frac{1}{2}\int_0^{\infty}\eta^{2s+n-1+4m}\cos\left(2\eta t^{\frac{3}{4}}+\eta^3t^{\frac{1}{4}}\right)\mathrm{e}^{-2\mu\eta^4}\mathrm{d}\eta\\
&=\frac{1}{2}(2\mu)^{-\frac{2s+n+4m}{4}}\int_0^{\infty}\tilde{\eta}^{2s+n-1+4m}\mathrm{e}^{-\tilde{\eta}^4}\mathrm{d}\tilde{\eta}+o(1)\\
&=\frac{1}{8}(2\mu)^{-\frac{2s+n+4m}{4}}\Gamma\left(\frac{2s+n}{4}+m\right)+o(1)
\end{align*}
for large time $t\gg1$, where we considered the Gamma function
\begin{align*}
\Gamma(z):=\int_0^{\infty}\eta^{z-1}\mathrm{e}^{-\eta}\mathrm{d}\eta=4\int_0^{\infty}\tilde{\eta}^{4z-1}\mathrm{e}^{-\tilde{\eta}^4}\mathrm{d}\tilde{\eta}. 
\end{align*}
Then, we are able to deduce
\begin{align}\label{Est-A-sum}
\sum\limits_{m=0,1,2}A_{2,m}(t)=\frac{1}{8}(2\mu)^{-\frac{2s+n+8}{4}}\Gamma\left(\frac{2s+n}{4}\right)\left[\,\left|2\mu P_{v_0}-\frac{2s+n}{32}P_{v_1}\right|^2+\frac{2s+n}{256}|P_{v_1}|^2\right]+o(1)
\end{align}
for large time $t\gg1$, where the well-known property of the Gamma function $\Gamma(z+1)=z\Gamma(z)$ was taken. In conclusion, recalling the definition of the combined data $v_c(x)$, from \eqref{Est-A1}-\eqref{Est-A-sum}, we derive the optimal lower bound estimate
\begin{align*}
	\|\,|\xi|^s\widehat{v}^{(2,p)}(t,\xi)\|_{L^2}^2\gtrsim t^{-\frac{2s+n}{4}}\left(|P_{v_1}|^2+|P_{v_c}|^2+|M_{v_1}|^2\right)
\end{align*}
for large time $t\gg1$. To end our proof, let us employ the Minkowski inequality to get
\begin{align*}
\|v(t,\cdot)-v^{(1,p)}(t,\cdot)\|_{\dot{H}^s}&\geqslant \|v^{(2,p)}(t,\cdot)\|_{\dot{H}^s}-\|v(t,\cdot)-v^{(1,p)}(t,\cdot)-v^{(2,p)}(t,\cdot)\|_{\dot{H}^s}\\
&\gtrsim t^{-\frac{2s+n}{8}}\big(|P_{v_1}|+|P_{v_c}|+|M_{v_1}|\big)
\end{align*}
for large time $t\gg1$, in which we applied the further error estimate \eqref{Est-Lin-Lead-02}.
\end{proof}
\begin{remark}\label{Rem-3.3}
In Theorem \ref{Thm-Linear-Leading-Term}, we investigate the optimal leading term, which is sharper than the profile derived in \cite[Theorem 1]{Su-Wang=2020} for the linearized dissipative Boussinesq equation \eqref{Eq-Linear-Dissipative-Boussinesq}. We may interpret it as the higher order diffusion waves, which is the reason for the growing phenomenon when $n=1,2$ in Theorem \ref{Thm-Linear-Opt-Est}. 
\end{remark}
\begin{remark}
By subtracting the optimal leading term $v^{(1,p)}$ in the $L^2$ framework, the optimal estimates in Theorem \ref{Thm-Linear-Opt-Est} can be improved. Particularly, this improvement is optimal provided that the new condition $\mb{A}_{\lin}\neq 0$ holds. We believe that this philosophy can be generalized to the linearized Boussinesq equation with structural damping in  \cite{Ikehata-Soga=2015}. Furthermore, due to an  improvement on the decay rate by subtracting the additional function $v^{(2,p)}$ in \eqref{Est-Lin-Lead-02}, we may explain it as a second  profile of solution for large time.
\end{remark}

\section{Global (in time) behavior of solution for the nonlinear model}\label{Section-GESDS}\setcounter{equation}{0}

\subsection{Philosophy of our proof for global (in time) existence}
$\ \ \ \ $Concerning any $T>0$, we now introduce the family $\{X_s(T)\}_{T>0}$ with $s\in[1,\infty)$ of evolution spaces
\begin{align*}
X_s(T):=\begin{cases}
\ml{C}([0,T],H^s\cap L^6)&\mbox{for}\ \ n=2\ \ \mbox{with}\ \ p=4 \ \ \mbox{when}\ \ s\in[1,3),\\
\ml{C}([0,T],H^s\cap L^6\cap L^{\infty})&\mbox{for}\ \ n=2\ \ \mbox{with}\ \ p=4 \ \ \mbox{when}\ \ s\in[3,\infty),\\
\ml{C}([0,T],H^s)&\mbox{otherwise},
\end{cases}
\end{align*}
carrying its norm excluding the special case $n=2$ with $p=4$ as follows:
\begin{align*}
\|u\|_{X_s(T)}:=\sup\limits_{t\in[0,T]}\left([\ml{D}_n(1+t)]^{-1}\|u(t,\cdot)\|_{L^2}+(1+t)^{\frac{2(s-1)+n}{8}}\|u(t,\cdot)\|_{\dot{H}^s}\right).
\end{align*}
and the norm for the limit case $n=2$ with $p=4$ as follows:
\begin{align*}
\|u\|_{X_s(T)}:=\begin{cases}
\sup\limits_{t\in[0,T]}\left([\ml{D}_2(1+t)]^{-1}\|u(t,\cdot)\|_{L^2}+(1+t)^{\frac{1}{6}}\|u(t,\cdot)\|_{L^6}+(1+t)^{\frac{s}{4}}\|u(t,\cdot)\|_{\dot{H}^s}\right.\\
\left.\qquad\quad +(1+t)^{\frac{1}{4}}\|u(t,\cdot)\|_{\dot{H}^1}\right)\ \ \mbox{when}\ \ s\in[1,3),\\ \\
\sup\limits_{t\in[0,T]}\left([\ml{D}_2(1+t)]^{-1}\|u(t,\cdot)\|_{L^2}+(1+t)^{\frac{1}{6}}\|u(t,\cdot)\|_{L^6}+(1+t)^{\frac{s}{4}}\|u(t,\cdot)\|_{\dot{H}^s}\right.\\
\left.\qquad\quad +(1+t)^{\frac{1}{4}}\|u(t,\cdot)\|_{\dot{H}^1}+(1+t)^{\frac{1}{4}}\|u(t,\cdot)\|_{L^{\infty}}\right)\ \ \mbox{when}\ \ s\in[3,\infty).
\end{cases}
\end{align*}
The above time-weighted Sobolev norms are strongly motivated by the optimal estimates of solution to the linearized model \eqref{Eq-Linear-Dissipative-Boussinesq}, precisely, \eqref{Est-Lin-Solution} and \eqref{Est-Lin-Derivative} with $s-1\geqslant 0$ in Theorem \ref{Thm-Linear-Opt-Est}, as well as the estimates in Corollary \ref{Coro-Linear-L6}. Note that the additional regularities of evolution spaces when $n=2$ with $p=4$ will contribute to overcoming the logarithmic growth $\ml{D}_2(1+t)$ of $\|u(t,\cdot)\|_{L^2}$ and retaining the optimal growth rate.

Let us recall the fundamental solutions $K_0=K_0(t,x)$ and $K_1=K_1(t,x)$, respectively, with initial data $(u_0,u_1)=(\delta_0,0)$ and $(u_0,u_1)=(0,\delta_0)$, studied in Section \ref{Section-Linearized-Model}. Here, $\delta_0$ denotes the Dirac distribution at $x=0$. Inspired by Duhamel's principle, we may introduce the operator $\Phi$ such that
\begin{align*}
\Phi:u(t,x)\in X_s(T)\to\Phi[u](t,x):=u^{\lin}(t,x)+u^{\non}(t,x),
\end{align*}
in which $u^{\lin}(t,x)\equiv v(t,x)$ is the solution to the corresponding linearized Cauchy problem \eqref{Eq-Linear-Dissipative-Boussinesq}, and the other term is defined via
\begin{align*}
u^{\non}(t,x):=\int_0^t K_1(t-\tau,x)\ast_{(x)}\Delta f(u(\tau,x);p)\mathrm{d}\tau.
\end{align*}
In order to prove the global (in time) existence of solution to the nonlinear dissipative Boussinesq equation \eqref{Eq-NonLinear-Dissipative-Boussinesq}, we shall demonstrate a fixed point of the operator to be the solution, namely, $\Phi [u]\in X_s(T)$ for all $T>0$. For this purpose, it suffices to verify the following crucial estimates:
\begin{align}
\|\Phi[u]\|_{X_s(T)}&\lesssim\|(u_0,u_1)\|_{\ml{A}_s}+\|u\|_{X_s(T)}^p,\label{Est-Crucial-01}\\
\|\Phi[u]-\Phi[\bar{u}]\|_{X_s(T)}&\lesssim\|u-\bar{u}\|_{X_s(T)}\left(\|u\|_{X_s(T)}^{p-1}+\|\bar{u}\|_{X_s(T)}^{p-1}\right),\label{Est-Crucial-02}
\end{align}
uniformly in $T>0$ with the initial data space \eqref{Data-Space}, under some appropriate conditions for $p$ related to $n$ and $s$. In the desire estimate \eqref{Est-Crucial-02}, $u$ and $\bar{u}$ are two solutions to the nonlinear dissipative Boussinesq equation \eqref{Eq-NonLinear-Dissipative-Boussinesq}. As a consequence, a standard application of Banach's fixed point theorem leads to global (in time) existence result of small data Sobolev solution $X_s(\infty)$. Remark that from Theorem \ref{Thm-Linear-Opt-Est}, Corollary \ref{Coro-Linear-Est} as well as Corollary \ref{Coro-Linear-L6}, the estimate $\|u^{\lin}\|_{X_s(T)}\lesssim\|(u_0,u_1)\|_{\ml{A}_s}$ holds. Note that we require $H^{1+\epsilon_1}$ regularity for $u_0$ only when $s\in[3,\infty)$ from Corollary \ref{Coro-Linear-L6}.
 Namely, our aim in \eqref{Est-Crucial-01} turns into 
\begin{align}\label{Est-Crucial-03}
\|u^{\non}\|_{X_s(T)}\lesssim\|u\|_{X_s(T)}^p.
\end{align}

To end this part, we collect some preliminaries from the harmonic analysis to estimate the nonlinear term in homogeneous Sobolev spaces (see also \cite[Appendix C]{Palmieri-2018}).
\begin{lemma}\label{fractionalgagliardonirenbergineq} (Fractional Gagliardo-Nirenberg inequality, \cite{Hajaiej-Molinet-Ozawa-Wang-2011})
	Let $p,p_0,p_1\in(1,\infty)$ and $\kappa\in[0,s)$ with $s\in(0,\infty)$. Then, for all $f\in L^{p_0}\cap \dot{H}^{s}_{p_1}$, the following inequality holds:
	\begin{align*}
		\|f\|_{\dot{H}^{\kappa}_{p}}\lesssim\|f\|_{L^{p_0}}^{1-\beta}\|f\|^{\beta}_{\dot{H}^{s}_{p_1}},
	\end{align*}
	where  $\beta=\left(\frac{1}{p_0}-\frac{1}{p}+\frac{\kappa}{n}\right)\big/\left(\frac{1}{p_0}-\frac{1}{p_1}+\frac{s}{n}\right)$ and $ \beta\in\left[\frac{\kappa}{s},1\right]$.
\end{lemma}

\begin{lemma}\label{fractionalchainrule} (Fractional chain rule, \cite{Palmieri-Reissig=2018})
	Let  $r,r_1,r_2\in(1,\infty)$ and $p\in(\lceil s\rceil,\infty)$  with $s\in(0,\infty)$ satisfying the relation
	\begin{equation*}
		\frac{1}{r}=\frac{p-1}{r_1}+\frac{1}{r_2}.
	\end{equation*}
	Then, for all $f\in\dot{H}^{s}_{r_2}\cap L^{r_1}$, the following inequality holds:
	\begin{equation*}
	\|\,|f|^p\|_{\dot{H}^{s}_{r}}\lesssim\|f\|_{L^{r_1}}^{p-1}\|f\|_{\dot{H}^{s}_{r_2}}.
	\end{equation*}
\end{lemma}

\begin{lemma}\label{fractionalpowersrule} (Fractional powers rule, \cite{Runst-Sickel=1996})
	Let $r\in(1,\infty)$ and $p\in(1,\infty)$ with $s\in(0,p)$. Then, for all $f\in{\dot{H}^{s}_{r}}\cap L^{\infty}$, the following inequality holds:
	\begin{equation*}
	\|\,|f|^p\|_{\dot{H}^{s}_{r}}\lesssim\|f\|_{\dot{H}^{s}_{r}}\|f\|_{L^{\infty}}^{p-1}.
	\end{equation*}
\end{lemma}
\begin{lemma}\label{coroleibunizpower}(Fractional Leibniz rule, \cite{Runst-Sickel=1996}) Let $r\in(1,\infty)$ and $p_1,p_2,q_1,q_2\in(1,\infty]$ with $s\in(0,\infty)$ satisfying the relation
\begin{align*}
\frac{1}{r}=\frac{1}{p_1}+\frac{1}{p_2}=\frac{1}{q_1}+\frac{1}{q_2}.
\end{align*}	
	Then, for all $f\in\dot{H}^{s}_{p_1}\cap L^{q_1}$ and $g\in\dot{H}^{s}_{q_2}\cap L^{p_2}$, the following inequality holds:
	\begin{equation*}
	\|fg\|_{\dot{H}^{s}_{r}}\lesssim \|f\|_{\dot{H}^{s}_{p_1}}\|g\|_{L^{p_2}}+\|f\|_{L^{q_1}}\|g\|_{\dot{H}^{s}_{q_2}}.
	\end{equation*}
\end{lemma}
\begin{lemma}\label{fractionembedd} (Fractional Sobolev embedding, \cite{Dabbicco-Ebert-Lucente-2017}) Let $0<2s^*<n<2s$. Then, for all $f\in\dot{H}^{s^*}\cap\dot{H}^s$, the following inequality holds:
	\begin{equation*}
		\|f\|_{L^{\infty}}\lesssim\|f\|_{\dot{H}^{s^*}}+\|f\|_{\dot{H}^s}.
	\end{equation*}
\end{lemma}

\subsection{Some estimates for the nonlinear term}
$\ \ \ \ $We next will derive some priori estimates for the power type nonlinearity $f(u(\tau,\cdot);p)$ in the $L^1$ and $\dot{H}^{(s-2)_+}$ spaces, respectively, by using the norm of evolution spaces $X_s(\tau)$. Let us denote $0<\varepsilon_1\ll 1$ when $n=2$ and $\varepsilon_1=0$ when $n\geqslant 3$ so that
\begin{align*}
\ml{D}_n(1+\tau)\lesssim (1+\tau)^{-\frac{n-2}{8}+\varepsilon_1}\ \ \mbox{for}\ \ n\geqslant 2.
\end{align*}

First of all, let us apply the fractional Gagliardo-Nirenberg inequality to get
\begin{align}\label{Est-Lm-Gag-Nir}
\|f(u(\tau,\cdot);p)\|_{L^m}&\lesssim\|u(\tau,\cdot)\|_{L^2}^{p-\frac{n}{s}(\frac{p}{2}-\frac{1}{m})}\|u(\tau,\cdot)\|_{\dot{H}^s}^{\frac{n}{s}(\frac{p}{2}-\frac{1}{m})}\notag\\
&\lesssim (1+\tau)^{-\frac{n-1}{4}p+\frac{n}{4m}+\varepsilon_1(\frac{2s-n}{2s}p+\frac{n}{sm})}\|u\|_{X_s(\tau)}^p,
\end{align}
where we restricted $\frac{2}{m}\leqslant p\leqslant\frac{2n}{m(n-2s)_+}$ with $m=1,2$. Note that $\dot{H}^{(s-2)_+}=L^2$ when $s\leqslant 2$.

We now turn to the case with $s>2$. For one thing, when $p>\lceil s-2\rceil$, by employing the fractional chain rule and the fractional Gagliardo-Nirenberg inequality, we arrive at
\begin{align}
\|f(u(\tau,\cdot);p)\|_{\dot{H}^{s-2}}&\lesssim\|u(\tau,\cdot)\|_{L^{r_1}}^{p-1}\|u(\tau,\cdot)\|_{\dot{H}^{s-2}_{r_2}}\notag\\
&\lesssim \|u(\tau,\cdot)\|_{L^2}^{\frac{2s-n}{2s}(p-1)+\frac{2}{s}}\|u(\tau,\cdot)\|_{\dot{H}^s}^{\frac{n}{2s}(p-1)+\frac{s-2}{s}}\notag\\
&\lesssim (1+\tau)^{-\frac{n-1}{4}p+\frac{n-2s+4}{8}+\varepsilon_1(\frac{2s-n}{2s}p+\frac{n-2s+4}{2s})}\|u\|_{X_s(\tau)}^p,\label{Est-Hs-01}
\end{align}
in which $\frac{p-1}{r_1}+\frac{1}{r_2}=\frac{1}{2}$ with $\frac{n}{s}(\frac{1}{2}-\frac{1}{r_1})\in[0,1]$ and $\frac{n}{s}(\frac{1}{2}-\frac{1}{r_2}+\frac{s-2}{n})\in[0,1]$. Hence, by straightforward computations, there exist two parameters $r_1,r_2\in(1,\infty)$ satisfying the last restrictions and it leads to $	1-\frac{2s-4}{n}\leqslant p \leqslant\frac{n-2s+4}{(n-2s)_+}$. For another, when $p>s-2$ with $2s>n$ additionally, we combine the fractional powers rule and the fractional Sobolev embedding to show
\begin{align*}
\|f(u(\tau,\cdot);p)\|_{\dot{H}^{s-2}}&\lesssim \|u(\tau,\cdot)\|_{\dot{H}^{s-2}}\|u(\tau,\cdot)\|_{L^{\infty}}^{p-1}\\
&\lesssim \|u(\tau,\cdot)\|_{L^2}^{\frac{2}{s}}\|u(\tau,\cdot)\|_{\dot{H}^s}^{\frac{s-2}{s}}\left(\|u(\tau,\cdot)\|_{\dot{H}^{s^*}}^{p-1}+\|u(\tau,\cdot)\|_{\dot{H}^s}^{p-1}\right)
\end{align*}
with $0<2s^*<n$. Again via the fractional Gagliardo-Nirenberg inequality, it results
\begin{align}
\|f(u(\tau,\cdot);p)\|_{\dot{H}^{s-2}}&\lesssim(1+\tau)^{-\frac{(2s^*+n-2)p+2s-4-2s^*}{8}+\varepsilon_1(\frac{s-s^*}{s}p+\frac{2-s+s^*}{s})}\|u\|_{X_s(\tau)}^p\notag\\
&\lesssim(1+\tau)^{-\frac{(2n-2-\varepsilon_2)p+2s-4-n+\varepsilon_2}{8}+\varepsilon_1(\frac{2s-n+\varepsilon_2}{2s}p+\frac{4-2s+n-\varepsilon_2}{2s})}\|u\|_{X_s(\tau)}^p\label{Est-Hs-02}
\end{align}
with a sufficiently small constant $0<\varepsilon_2\ll 1$ by choosing $s^*$ such that $2s^*=n-\varepsilon_2$.

To overcome some difficulties in the limit case $n=2$ with $p=4$, we suggest some estimates for the nonlinearity. For one thing, via the classical interpolation and the chain rule, we know
\begin{align}\label{EST-01}
\|\,|D|f(u(\tau,\cdot);4)\|_{L^1}+\|f(u(\tau,\cdot);4)\|_{L^2}&\lesssim\|u(\tau,\cdot)\|_{\dot{H}^1}\|u(\tau,\cdot)\|_{L^6}^3\lesssim (1+\tau)^{-\frac{3}{4}}\|u\|_{X_s(\tau)}^4.
\end{align}
For another, concerning $s\in[3,\infty)$, one has
\begin{align*}
\|f(u(\tau,\cdot);4)\|_{\dot{H}^{s-1}}&\lesssim\|u(\tau,\cdot)\|_{L^{\infty}}^3\|u(\tau,\cdot)\|_{\dot{H}^{s-1}}\\
&\lesssim\|u(\tau,\cdot)\|_{L^{\infty}}^3\|u(\tau,\cdot)\|_{\dot{H}^1}^{\frac{1}{s-1}}\|u(\tau,\cdot)\|_{\dot{H}^s}^{\frac{s-2}{s-1}}\\
&\lesssim (1+\tau)^{-\frac{1}{2}-\frac{s}{4}}\|u\|_{X_s(\tau)}^4,
\end{align*}
and similarly,
\begin{align*}
\|f(u(\tau,\cdot);4)\|_{\dot{H}^{s-2}}\lesssim (1+\tau)^{-\frac{1}{4}-\frac{s}{4}}\|u\|_{X_s(\tau)}^4.
\end{align*}

\subsection{Proof of Theorem \ref{Thm-GESDS} with different regularities}
\subsubsection{Global (in time) solution with lower regularity data: $s\in[1,2]$}
$\ \ \ \ $Let us start the proof of \eqref{Est-Crucial-03} excluding the case $n=2$ with $p=4$ by applying the derived $(L^2\cap L^1)-L^2$ type estimate in $[0,t/2]$ and $L^2-L^2$ type estimate in $[t/2,t]$ from Corollary \ref{Coro-Linear-Est} carrying $s=0$, namely,
\begin{align}
\|u^{\non}(t,\cdot)\|_{L^2}&=\left\|\int_0^t\Delta K_1(t-\tau,\cdot)\ast_{(x)}f(u(\tau,\cdot);p)\mathrm{d}\tau\right\|_{L^2}\notag\\
&\lesssim \int_0^{t/2}(1+t-\tau)^{-\frac{n+2}{8}}\|f(u(\tau,\cdot);p)\|_{L^2\cap L^1}\mathrm{d}\tau+\int_{t/2}^t(1+t-\tau)^{-\frac{1}{4}}\|f(u(\tau,\cdot);p)\|_{L^2}\mathrm{d}\tau\notag\\
&\lesssim (1+t)^{-\frac{n+2}{8}}\int_0^{t/2}(1+\tau)^{-\frac{n-1}{4}p+\frac{n}{4}+\varepsilon_1(\frac{2s-n}{2s}p+\frac{n}{s})}\mathrm{d}\tau\|u\|_{X_s(T)}^p\notag\\
&\quad+(1+t)^{-\frac{n-1}{4}p+\frac{n}{8}+\varepsilon_1(\frac{2s-n}{2s}p+\frac{n}{2s})}\int_{t/2}^t(1+t-\tau)^{-\frac{1}{4}}\mathrm{d}\tau\|u\|_{X_s(T)}^p,\label{Est-4}
\end{align}
where we used \eqref{Est-Lm-Gag-Nir} with $m=1,2$ associated with the condition $2\leqslant p\leqslant\frac{n}{(n-2s)_+}$. Note that $\|u\|_{X_s(\tau)}\leqslant\|u\|_{X_s(T)}$ for all $\tau\in[0,T]$. Here, the asymptotics $(1+t-\tau)\approx (1+t)$ when $\tau\in[0,t/2]$ and $(1+\tau)\approx (1+t)$ when $\tau\in[t/2,t]$ were considered. Thanks to our crucial assumptions $p>4$ when $n=2$ and $p\geqslant\frac{n+2}{n-1}$ when $n\geqslant 3$, we are able to claim
\begin{align}\label{Est-Sub-01}
[\ml{D}_n(1+t)]^{-1}\|u^{\non}(t,\cdot)\|_{L^2}\lesssim \|u\|_{X_s(T)}^p.
\end{align}
Similarly, under the same restriction on the exponent $p$, we may obtain
\begin{align*}
(1+t)^{\frac{2(s-1)+n}{8}}\|u^{\non}(t,\cdot)\|_{\dot{H}^s}&\lesssim(1+t)^{\frac{2(s-1)+n}{8}}\int_0^{t/2}(1+t-\tau)^{-\frac{2(s+1)+n}{8}}\|f(u(\tau,\cdot);p)\|_{L^2\cap L^1}\mathrm{d}\tau\\
&\quad+(1+t)^{\frac{2(s-1)+n}{8}}\int_{t/2}^t(1+t-\tau)^{-\frac{s+1}{4}}\|f(u(\tau,\cdot);p)\|_{L^2}\mathrm{d}\tau\\
&\lesssim \|u\|_{X_s(T)}^p,
\end{align*}
where the inequality \eqref{Est-Lm-Gag-Nir} with $m=1,2$ was employed again. Summarizing the last two estimates, we derive the desire estimate \eqref{Est-Crucial-03} when $s\in[1,2]$.

Let us turn to the limit case $n=2$ with $p=4$. An application of the derived estimates in Corollary \ref{Coro-New} and \eqref{EST-01} shows
\begin{align*}
\|u^{\non}(t,\cdot)\|_{L^2}&=\left\|\int_0^t|D|K_1(t-\tau,\cdot)\ast_{(x)}|D|f(u(\tau,\cdot);4)\mathrm{d}\tau\right\|_{L^2}\\
&\lesssim\int_0^t(1+t-\tau)^{-\frac{1}{4}}\|\,|D|f(u(\tau,\cdot);4)\|_{L^1}\mathrm{d}\tau+\int_0^t\mathrm{e}^{-c(t-\tau)}\|f(u(\tau,\cdot);4)\|_{L^2}\mathrm{d}\tau\\
&\lesssim\int_0^t(1+t-\tau)^{-\frac{1}{4}}(1+\tau)^{-\frac{3}{4}}\mathrm{d}\tau\|u\|_{X_s(T)}^4+\int_0^t\mathrm{e}^{-c(t-\tau)}(1+\tau)^{-\frac{3}{4}}\mathrm{d}\tau\|u\|_{X_s(T)}^4\\
&\lesssim\|u\|_{X_s(T)}^4\lesssim\ml{D}_2(1+t)\|u\|_{X_s(T)}^4.
\end{align*}
Via an analogous way, due to $s+1\in[2,3]$, one gets
\begin{align*}
&(1+t)^{\frac{1}{6}}\|u^{\non}(t,\cdot)\|_{L^6}+(1+t)^{\frac{1}{4}}\|u^{\non}(t,\cdot)\|_{\dot{H}^1}+(1+t)^{\frac{s}{4}}\|u^{\non}(t,\cdot)\|_{\dot{H}^s}\\
&\qquad\lesssim\left((1+t)^{\frac{1}{6}}\int_0^t(1+t-\tau)^{-\frac{5}{12}}(1+\tau)^{-\frac{3}{4}}\mathrm{d}\tau+(1+t)^{\frac{1}{4}}\int_0^t(1+t-\tau)^{-\frac{1}{2}}(1+\tau)^{-\frac{3}{4}}\mathrm{d}\tau\right.\\
&\qquad\qquad\ \left.+(1+t)^{\frac{s}{4}}\int_0^t(1+t-\tau)^{-\frac{s+1}{4}}(1+\tau)^{-\frac{3}{4}}\mathrm{d}\tau+(1+t)^{\frac{s}{4}}\int_0^t\mathrm{e}^{-c(t-\tau)}(1+\tau)^{-\frac{3}{4}}\mathrm{d}\tau\right)\|u\|_{X_s(T)}^4\\
&\qquad\lesssim\|u\|_{X_s(T)}^4.
\end{align*}
Summing up the last estimates, we conclude the desired estimate \eqref{Est-Crucial-03} when $n=2$ with $p=4$.

In order to demonstrate \eqref{Est-Crucial-02}, we notice that
\begin{align*}
\|\Phi[u]-\Phi[\bar{u}]\|_{X_s(T)}=\left\|\int_0^t\Delta K_1(t-\tau,\cdot)\ast_{(x)}[f(u(\tau,\cdot);p)-f(\bar{u}(\tau,\cdot);p)]\mathrm{d}\tau\right\|_{X_s(T)}.
\end{align*}
Thanks to H\"older's inequality, the source error term can be controlled by
\begin{align*}
\|f(u(\tau,\cdot);p)-f(\bar{u}(\tau,\cdot);p)\|_{L^m}\lesssim \|u(\tau,\cdot)-\bar{u}(\tau,\cdot)\|_{L^{mp}}\left(\|u(\tau,\cdot)\|_{L^{mp}}^{p-1}+\|\bar{u}(\tau,\cdot)\|_{L^{mp}}^{p-1}\right)
\end{align*}
with $m=1,2$. Finally, employing the fractional Gagliardo-Nirenberg inequality as those in \eqref{Est-Lm-Gag-Nir}, we can complete the estimate \eqref{Est-Crucial-02}. Therefore, our proof is finished when $s\in[1,2]$.

\subsubsection{Global (in time) solution with suitable regularity data: $s\in(2,\frac{n}{2}]$}
$\ \ \ \ $Due to the failure of the Sobolev embedding $ H^s\hookrightarrow L^{\infty}$ when $s\leqslant \frac{n}{2}$, we turn to the intermediate case $s\in(2,\frac{n}{2}]$. Remark that we just consider the situation $n\geqslant 5$ to guarantee the non-empty set of $s$. It allows us to employ the fractional chain rule to estimate the nonlinearity in the homogeneous Sobolev spaces. 

Concerning $u^{\non}(t,\cdot)$ in the $L^2$ norm, the estimate \eqref{Est-Sub-01} still holds with the same deduction as the one for $s\in[1,2]$, where we required $\max\{2,\frac{n+2}{n-1}\}\leqslant p\leqslant \frac{n}{(n-2s)_+}$ when $n\geqslant 5$. For another, concerning the higher regularity, by  using Corollary \ref{Coro-Linear-Est}, \eqref{Est-Lm-Gag-Nir} with $m=1$ and \eqref{Est-Hs-01}, we may derive
\begin{align}\label{Est-5}
\|u^{\non}(t,\cdot)\|_{\dot{H}^s}&\lesssim\int_0^{t/2}(1+t-\tau)^{-\frac{n+2+2s}{8}}\|f(u(\tau,\cdot);p)\|_{\dot{H}^{s-2}\cap L^1}\mathrm{d}\tau\notag\\
&\quad+\int_{t/2}^t(1+t-\tau)^{-\frac{3}{4}}\|f(u(\tau,\cdot);p)\|_{\dot{H}^{s-2}}\mathrm{d}\tau\notag\\
&\lesssim (1+t)^{-\frac{n+2+2s}{8}}\int_0^{t/2}(1+\tau)^{-\frac{n-1}{4}p+\frac{n}{4}}\mathrm{d}\tau\|u\|_{X_s(T)}^p\notag\\
&\quad+(1+t)^{-\frac{n-1}{4}p+\frac{n-2s+4}{8}}\int_{t/2}^t(1+t-\tau)^{-\frac{3}{4}}\mathrm{d}\tau\|u\|_{X_s(T)}^p,
\end{align}
where we restricted $2\leqslant p\leqslant \frac{n-2s+4}{(n-2s)_+}$ as well as $p>\lceil s-2\rceil$. Then, since $n\geqslant 5$, it gives
\begin{align}\label{Est-0}
(1+t)^{\frac{2(s-1)+n}{8}}\|u^{\non}(t,\cdot)\|_{\dot{H}^s}\lesssim \|u\|_{X_s(T)}^p,
\end{align}
and we already obtained the aim estimate \eqref{Est-Crucial-03} when $s\in(2,\frac{n}{2}]$.

Our next step is to derive the Lipschitz condition for the uniqueness. Clearly, from H\"older's inequality and the fractional Gagliardo-Nirenberg inequality, the following estimate holds:
\begin{align*}
[\ml{D}_n(1+t)]^{-1}\|u^{\non}(t,\cdot)-\bar{u}^{\non}(t,\cdot)\|_{L^2}\lesssim\|u-\bar{u}\|_{X_s(T)}\left(\|u\|_{X_s(T)}^{p-1}+\|\bar{u}\|_{X_s(T)}^{p-1}\right).
\end{align*}
Actually, the subsequent key part is to estimate the source error term in the homogeneous Sobolev space $\dot{H}^{s-2}$, in other words,
\begin{align*}
&(1+t)^{\frac{2(s-1)+n}{8}}\|u^{\non}(t,\cdot)-\bar{u}^{\non}(t,\cdot)\|_{\dot{H}^s}\\
&\qquad\lesssim (1+t)^{-\frac{1}{2}}\int_0^{t/2}\|f(u(\tau,\cdot);p)-f(\bar{u}(\tau,\cdot);p)\|_{\dot{H}^{s-2}\cap L^1}\mathrm{d}\tau\\
&\qquad\quad+(1+t)^{\frac{2(s-1)+n}{8}}\int_{t/2}^t(1+t-\tau)^{-\frac{3}{4}}\|f(u(\tau,\cdot);p)-f(\bar{u}(\tau,\cdot);p)\|_{\dot{H}^{s-2}}\mathrm{d}\tau.
\end{align*}
We set $G(u):=|u|^{p-2}u$ and apply the fact
\begin{align}\label{Est-1}
|f(u(\tau,x);p)-f(\bar{u}(\tau,x);p)|&\approx \big||u(\tau,x)|^p-|\bar{u}(\tau,x)|^p\big|\notag\\
&\lesssim|u(\tau,x)-\bar{u}(\tau,x)|\int_0^1\big|G\big(\gamma u(\tau,x)+(1-\gamma)\bar{u}(\tau,x)\big)\big|\mathrm{d}\gamma.
\end{align}
As a consequence, one has
\begin{align*}
\|f(u(\tau,\cdot);p)-f(\bar{u}(\tau,\cdot);p)\|_{\dot{H}^{s-2}}&\lesssim\|u(\tau,\cdot)-\bar{u}(\tau,\cdot)\|_{\dot{H}^{s-2}_{r_3}}\int_0^1\big\|G\big(\gamma u(\tau,\cdot)+(1-\gamma)\bar{u}(\tau,\cdot)\big)\big\|_{L^{r_4}}\mathrm{d}\gamma\\
&\quad+\|u(\tau,\cdot)-\bar{u}(\tau,\cdot)\|_{L^{r_5}}\int_0^1\big\|G\big(\gamma u(\tau,\cdot)+(1-\gamma)\bar{u}(\tau,\cdot)\big)\big\|_{\dot{H}^{s-2}_{r_6}}\mathrm{d}\gamma,
\end{align*}
where we applied the fractional Leibniz rule with $\frac{1}{r_3}+\frac{1}{r_4}=\frac{1}{r_5}+\frac{1}{r_6}=\frac{1}{2}$. Note that
\begin{align*}
\int_0^1\big\|G\big(\gamma u(\tau,\cdot)+(1-\gamma)\bar{u}(\tau,\cdot)\big)\big\|_{L^{r_4}}\mathrm{d}\gamma\lesssim \|u(\tau,\cdot)\|_{L^{r_4(p-1)}}^{p-1}+\|\bar{u}(\tau,\cdot)\|_{L^{r_4(p-1)}}^{p-1}.
\end{align*}
For another term, taking $\frac{1}{r_6}=\frac{p-2}{r_7}+\frac{1}{r_8}$ in the fractional chain rule, we may arrive at
\begin{align*}
\big\|G\big(\gamma u(\tau,\cdot)+(1-\gamma)\bar{u}(\tau,\cdot)\big)\big\|_{\dot{H}^{s-2}_{r_6}}&\lesssim\|\gamma u(\tau,\cdot)+(1-\gamma)\bar{u}(\tau,\cdot)\|_{L^{r_7}}^{p-2}\|\gamma u(\tau,\cdot)+(1-\gamma)\bar{u}(\tau,\cdot)\|_{\dot{H}^{s-2}_{r_8}}
\end{align*}
carrying the additional restriction $p>1+\lceil s-2\rceil$. By lengthy but straightforward computations with the fractional Gagliardo-Nirenberg inequality, we can derive 
\begin{align}\label{Est-3}
	(1+t)^{\frac{2(s-1)+n}{8}}\|u^{\non}(t,\cdot)-\bar{u}^{\non}(t,\cdot)\|_{\dot{H}^s}\lesssim\|u-\bar{u}\|_{X_s(T)}\left(\|u\|_{X_s(T)}^{p-1}+\|\bar{u}\|_{X_s(T)}^{p-1}\right)
\end{align}
when $p>1+\lceil s-2\rceil$. Then, the estimate \eqref{Est-Crucial-02} is completed as $s\in(2,\frac{n}{2}]$. Collecting the considered restrictions on the exponent $p$ when $s\in(2,\frac{n}{2}]$ as well as $n\geqslant 5$ such that
\begin{align*}
2=\max\left\{2,\frac{n+2}{n-1}\right\}\leqslant p\leqslant \frac{n}{(n-2s)_+}, \ \ 1<p\leqslant\frac{n-2s+4}{(n-2s)_+},\ \ \mbox{and}\ \ p>\lceil s-2\rceil+1,
\end{align*}
we conclude the final restriction $\lceil s-2\rceil+1<p\leqslant\frac{n-2s+4}{(n-2s)_+}$.

\subsubsection{Global (in time) solution with large regularity data: $s\in(\frac{n}{2},\infty)$}
$\ \ \ \ $In the large regularity situation, we may apply the fractional Sobolev embedding. Let us firstly start the proof excluding the case $n=2$ with $p=4$. Concerning $u^{\non}(t,\cdot)$ in the $L^2$ norm, the estimate \eqref{Est-Sub-01} still holds. For another, concerning the higher regularity of solution, by  using Corollary \ref{Coro-Linear-Est}, \eqref{Est-Lm-Gag-Nir} with $m=1$ and \eqref{Est-Hs-02}, we are able to deduce
\begin{align}
	\|u^{\non}(t,\cdot)\|_{\dot{H}^s}
	&\lesssim (1+t)^{-\frac{n+2+2s}{8}}\int_0^{t/2}(1+\tau)^{-\frac{n-1}{4}p+\frac{n}{4}+\varepsilon_1(\frac{2s-n}{2s}p+\frac{n}{s})}\mathrm{d}\tau\|u\|_{X_s(T)}^p\notag\\
	&\quad+(1+t)^{-\frac{(2n-2-\varepsilon_2)p+2s-4-n+\varepsilon_2}{8}+\varepsilon_1(\frac{2s-n+\varepsilon_2}{2s}p+\frac{4-2s+n-\varepsilon_2}{2s})+\frac{1}{4}}\|u\|_{X_s(T)}^p,\label{Est-6}
\end{align}
where we restricted $p\geqslant 2$ as well as $p>s-2$. Our assumption $p\geqslant\frac{n+2}{n-1}$ when $n\geqslant 2$ implies the estimate \eqref{Est-0}. Therefore, our aim \eqref{Est-Crucial-03} is completed excluding the special case.

For the limit case $n=2$ with $p=4$, the philosophies of the proof between $s\in[1,2]$ and $s\in[2,3)$ are the same. Thus, we only treat the situation for $s\in[3,\infty)$. The $L^2$, $L^6$ and $\dot{H}^1$ norms for $u^{\non}(t,\cdot)$ can be obtained by the same manner as the previous situations. It remains to study $u^{\non}(t,\cdot)$ for the $L^{\infty}$ and $\dot{H}^s$ norms. With the aid of \eqref{EST-01} and Corollary \ref{Coro-New}, it holds
\begin{align*}
\|u^{\non}(t,\cdot)\|_{L^{\infty}}&\lesssim\int_0^t(1+t-\tau)^{-\frac{1}{2}}(1+\tau)^{-\frac{3}{4}}\mathrm{d}\tau\|u\|_{X_s(T)}^4+\int_0^t\mathrm{e}^{-c(t-\tau)}(1+\tau)^{-\frac{3}{4}}\mathrm{d}\tau\|u\|_{X_s(T)}^4\\
&\lesssim (1+t)^{-\frac{1}{4}}\|u\|_{X_s(T)}^4.
\end{align*}
Lastly, by applying the interpolation for $\dot{H}^{s-1}$ additionally in $[t/2,t]$, we find
\begin{align*}
\|u^{\non}(t,\cdot)\|_{\dot{H}^s}&\lesssim\int_0^{t/2}(1+t-\tau)^{-\frac{s+1}{4}}\|\,|D|f(u(\tau,\cdot);4)\|_{L^1}\mathrm{d}\tau+\int_{t/2}^t(1+t-\tau)^{-\frac{1}{2}}\|f(u(\tau,\cdot);4)\|_{\dot{H}^{s-1}}\mathrm{d}\tau\\
&\quad+\int_0^t\mathrm{e}^{-c(t-\tau)}\|f(u(\tau,\cdot);4)\|_{\dot{H}^{s-2}}\mathrm{d}\tau\\
&\lesssim\int_0^{t/2}(1+t-\tau)^{-\frac{s+1}{4}}(1+\tau)^{-\frac{3}{4}}\mathrm{d}\tau\|u\|_{X_s(T)}^4+\int_{t/2}^t(1+t-\tau)^{-\frac{1}{2}}(1+\tau)^{-\frac{1}{2}-\frac{s}{4}}\mathrm{d}\tau\|u\|_{X_s(T)}^4\\
&\quad+\int_0^t\mathrm{e}^{-c(t-\tau)}(1+\tau)^{-\frac{1}{4}-\frac{s}{4}}\mathrm{d}\tau\|u\|_{X_s(T)}^4\\
&\lesssim(1+t)^{-\frac{s}{4}}\|u\|_{X_s(T)}^4,
\end{align*}
which finishes the proof of \eqref{Est-Crucial-01} under this situation.

To verify the Lipschitz condition \eqref{Est-Crucial-02}, taking the additional condition $p>s-1$ and using \eqref{Est-1}, we may claim
\begin{align*}
	\|f(u(\tau,\cdot);p)-f(\bar{u}(\tau,\cdot);p)\|_{\dot{H}^{s-2}}&\lesssim\|u(\tau,\cdot)-\bar{u}(\tau,\cdot)\|_{\dot{H}^{s-2}}\int_0^1\big\|G\big(\gamma u(\tau,\cdot)+(1-\gamma)\bar{u}(\tau,\cdot)\big)\big\|_{L^{\infty}}\mathrm{d}\gamma\\
	&\quad+\|u(\tau,\cdot)-\bar{u}(\tau,\cdot)\|_{L^{\infty}}\int_0^1\big\|G\big(\gamma u(\tau,\cdot)+(1-\gamma)\bar{u}(\tau,\cdot)\big)\big\|_{\dot{H}^{s-2}}\mathrm{d}\gamma,
\end{align*}
moreover,
\begin{align*}
\big\|G\big(\gamma u(\tau,\cdot)+(1-\gamma)\bar{u}(\tau,\cdot)\big)\big\|_{L^{\infty}}&\lesssim\|u(\tau,\cdot)\|_{L^{\infty}}^{p-1}+\|\bar{u}(\tau,\cdot)\|_{L^{\infty}}^{p-1},\\
\big\|G\big(\gamma u(\tau,\cdot)+(1-\gamma)\bar{u}(\tau,\cdot)\big)\big\|_{\dot{H}^{s-2}}&\lesssim \|\gamma u(\tau,\cdot)+(1-\gamma)\bar{u}(\tau,\cdot)\|_{\dot{H}^{s-2}}\|\gamma u(\tau,\cdot)+(1-\gamma)\bar{u}(\tau,\cdot)\|_{L^{\infty}}^{p-2}.
\end{align*}
Lastly, by using the fractional Sobolev embedding and the fractional Gagliardo-Nirenberg inequality, we may derive \eqref{Est-3}. Our proof is totally complete.

\subsection{Proof of Theorem \ref{Thm-Global-Profiles}}
$\ \ \ \ $As a preparation of studying asymptotic behavior of global (in time) solution for the nonlinear dissipative Boussinesq equation \eqref{Eq-NonLinear-Dissipative-Boussinesq}, we introduce the next preliminary.  According to the global (in time) existence result proved in Section \ref{Section-GESDS}, under the hypotheses in Theorem \ref{Thm-GESDS}, the nonlinearity fulfills the estimates \eqref{Est-Lm-Gag-Nir}-\eqref{Est-Hs-02}. Moreover, we already derived $\|u\|_{X_s(T)}\lesssim\|(u_0,u_1)\|_{\ml{A}_s}$ uniformly in time $T>0$. The next estimate is obtained in the proof of Theorem \ref{Thm-GESDS}. So, we omit the detail.
\begin{prop}\label{Prop-Nonlinear-Refined}
	Let the exponent $p$, the regularity parameter $s$ and initial data $u_0,u_1$ satisfying the same hypotheses as those in Theorem \ref{Thm-GESDS}. Additionally, we assume $p>\frac{n+2}{n-1}$ for $n\geqslant 2$. Then, the global (in time) Sobolev solution introduced in Theorem \ref{Thm-GESDS} fulfills the following refined estimates:
	\begin{align*}
		\left\|\int_0^tK_1(t-\tau,\cdot)\ast_{(x)}\Delta f(u(\tau,\cdot);p)\mathrm{d}\tau\right\|_{\dot{H}^s}=o(\ml{B}_{n,s}(t))
	\end{align*}
	for large time $t\gg1$, where the time-dependent function $\ml{B}_{n,s}(t)$ is denoted in \eqref{Bns}.
\end{prop}

Thanks to Theorem \ref{Thm-GESDS}, the global (in time) solution
\begin{align*}
u(t,x)=\sum\limits_{j=0,1}K_j(t,x)\ast_{(x)}u_j(x)+\int_0^tK_1(t-\tau,x)\ast_{(x)}\Delta f(u(\tau,x);p)\mathrm{d}\tau
\end{align*}
fulfills the following estimates:
\begin{align}\label{Est-new-02}
\left\|u(t,\cdot)-u^{(1,p)}(t,\cdot)\right\|_{\dot{H}^s}&\leqslant\left\|K_0(t,\cdot)\ast_{(x)}u_0(\cdot)+K_1(t,\cdot)\ast_{(x)}u_1(\cdot)-\ml{G}_1(t,\cdot)P_{u_1}\right\|_{\dot{H}^s}+o(\ml{B}_{n,s}(t))\notag\\
&=o(\ml{B}_{n,s}(t))
\end{align}
for large time $t\gg1$, where we used \eqref{Est-new} and Proposition \ref{Prop-Nonlinear-Refined}.

Furthermore, by applying the Minkowski inequality associated with \eqref{Est-new-02}, it gives
\begin{align*}
\|u(t,\cdot)\|_{\dot{H}^s}
\gtrsim\|\ml{G}_1(t,\cdot)\|_{\dot{H}^s}|P_{u_1}|-\left\|u(t,\cdot)-u^{(1,p)}(t,\cdot)\right\|_{\dot{H}^s}\gtrsim\ml{B}_{n,s}(t)|P_{u_1}|
\end{align*}
for large time $t\gg1$, provided $|P_{u_1}|\neq0$. Our proof is completed.

\section{Final remarks}
$\ \ \ \ $Throughout this paper, we have derived optimal estimates and large time asymptotic profile of global (in time) solution to the nonlinear dissipative Boussinesq equation \eqref{Eq-NonLinear-Dissipative-Boussinesq} by deep analysis on the corresponding linearized model. We believe that our philosophies can be applied in studying sharp large time behaviors for the Boussinesq equation with other damping terms, e.g. the general structural damping, namely,
\begin{align*}
	\begin{cases}
		u_{tt}-\Delta u+\Delta^2u+2\mu(-\Delta)^{\theta}u_t=\Delta f(u;p),&x\in\mb{R}^n,\ t>0,\\
		u(0,x)=u_0(x),\ u_t(0,x)=u_1(x),&x\in\mb{R}^n,
	\end{cases}
\end{align*}
with $\theta\in[0,2]$, in which the limit case $\theta=2$ has been completed in the manuscript. Moreover, the influence of dispersion, dissipation and nonlinearity on asymptotic behavior of global (in time) solution is still interesting. In our forthcoming works, we will answer the above equations even in the $L^q$ framework with $q\in[1,\infty]$.

\section*{Acknowledgments}
Wenhui Chen is supported in part by the National Natural Science Foundation of China (grant No. 12301270, No. 12171317) and Guangdong Basic and Applied Basic Research Foundation (grant No. 2023A1515012044).  Hiroshi Takeda is supported in part by the Grant-in-Aid for Scientific Research (C)
(grant No. 19K03596) from Japan Society for the Promotion of Science.

\end{document}